\documentclass[11pt]{amsart}

\usepackage{amsfonts, amstext, amsmath, amsthm, amscd, amssymb}
\usepackage{epsfig, graphics, color}
\usepackage{import}

\let\oldmarginpar\marginpar
\renewcommand\marginpar[1]{\oldmarginpar[\raggedleft\footnotesize #1]%
{\raggedright\footnotesize #1}}

\renewcommand{\setminus}{{\smallsetminus}}

\newcommand{\bdy}{{\partial}}

\newcommand{\tw}{{\rm tw}}

\def\co{\colon\thinspace}

\theoremstyle{plain}
\newtheorem{theorem}{Theorem}[section]

\newtheorem{lemma}[theorem]{Lemma}

\newtheorem{prop}[theorem]{Proposition}

\newtheorem*{namedtheorem}{\theoremname}
\newcommand{\theoremname}{testing}

\theoremstyle{definition}
\newtheorem{define}[theorem]{Definition}

\begin{document}
\title{Essential twisted surfaces in alternating link complements}
\author{Marc Lackenby}
\address{Mathematical Institute, University of Oxford, Oxford, OX2 6GG, UK}

\author{Jessica S. Purcell}
\address{School of Mathematical Sciences, Monash University, Clayton, VIC 3800, Australia}



\begin{abstract}
Checkerboard surfaces in alternating link complements are used
frequently to determine information about the link.  However, when
many crossings are added to a single twist region of a link diagram,
the geometry of the link complement stabilizes (approaches a
geometric limit), but a corresponding checkerboard surface increases
in complexity with crossing number.  In this paper, we generalize
checkerboard surfaces to certain immersed surfaces, called twisted
checkerboard surfaces, whose geometry better reflects that of the
alternating link in many cases.  We describe the surfaces,
show that they are essential in the complement of an alternating link,
and discuss their properties, including an analysis of homotopy
classes of arcs on the surfaces in the link complement.
\end{abstract}

\maketitle

\section{Introduction}
Essential surfaces in link complements have played an important role in geometric topology and knot theory.  The checkerboard surfaces in alternating links are particularly important.  They have been used to analyze volumes \cite{lackenby:volume-alt}, to obtain singular structures \cite{alr:alternating}, and to give a polyhedral decomposition \cite{menasco:polyhedra}, among other things.  Menasco and Thistlethwaite proved they are incompressible and boundary incompressible \cite{menasco-thist:classif-alt}.

The genus of a checkerboard surface is determined by the crossing number of the diagram.  When more and more crossings are added to a single twist region of a diagram of a hyperbolic link, the genus of a corresponding checkerboard surface increases without bound, while the link complement approaches a geometric limit.  For this reason, checkerboard surfaces are not always ideally suited for analyzing geometric properties of a hyperbolic link complement.

In this paper we generalize checkerboard surfaces in alternating link complements to another class of surfaces, which we call twisted checkerboard surfaces.  These surfaces are immersed in the link complement rather than embedded, but they capture the geometry of the link complement in useful ways when the link has many crossings in some twist regions.  These surfaces feature prominently in our recent proof that alternating knots have cusp volume bounded below by a linear function of the twist number of the knot \cite{lp:acv}.

The main result of this paper is to show that these surfaces are essential.  We also analyze geometric and homotopic properties of these surfaces.  For example, we determine when two distinct arcs on the surface will be homotopic in the link complement.

To define the surfaces and state our results precisely, we recall some definitions.

\subsection{Definitions}


A diagram of a link $K$ is said to be \emph{prime} if, for each simple closed curve $\gamma$ that lies on the plane of projection and meets the diagram transversely exactly twice in the interiors of edges, the curve $\gamma$ bounds on one side a portion of the diagram with no crossings.

Menasco showed that any nonsplit, prime, alternating diagram specifies either a hyperbolic link or a $(2,q)$--torus link \cite{menasco}.  We will be concerned only with hyperbolic alternating links in this paper.  

For an alternating link, we define a \emph{twist region} to be a string of bigons arranged end to end in the diagram graph, which is maximal in that there are no additional bigons on either end.  A single crossing adjacent to no bigons is also defined to be a twist region.

When we consider diagrams of alternating links, we often want them to have as few twist regions as possible, in the following sense.

\begin{define}\label{def:twist-reduced}
A diagram is \emph{twist reduced} if any simple closed curve that meets the diagram graph in exactly two vertices  and that, at each crossing, runs between opposite regions, encloses a string of bigons of the diagram on one side.  (See, for example, \cite[Figure~3]{lackenby:volume-alt}.)
\end{define}

Suppose a simple closed curve $\gamma$ in the projection plane meets the diagram in exactly two vertices and that, at each crossing, runs between opposite regions.  By sliding $\gamma$ to contain both vertices on one side, and then applying a flype to the other side, we can either remove both crossings, or move one of the crossings to be in the same twist region as the other.  Thus every alternating link has a twist reduced alternating diagram.

\begin{define}\label{def:twist-number}
The \emph{twist number} of an alternating link is the number of twist regions in a twist reduced diagram.  We denote the twist number of the link $K$ by $\tw(K)$.
\end{define}

The twist number of an alternating knot is an invariant of the knot. For example, this follows by the invariance of characteristic squares under flyping \cite{lackenby:volume-alt} along with the solution of the Tait flyping conjecture \cite{menasco-thist:classif-alt}, or by relating twist number to the Jones polynomial as in \cite{dasbach-lin:volumish}.

\subsection{Twisted surfaces}

In this subsection, we will define the twisted checkerboard surfaces. First, fix a prime, twist reduced, alternating diagram of the hyperbolic alternating link $K$.  Throughout, we will abuse notation and refer to the link and its diagram by $K$.

For each twist region of $K$ with at least $N_\tw$ crossings, where $N_\tw$ will be determined later, we will \emph{augment} the diagram to obtain a new link diagram, as in Figure \ref{fig:links-ex}.  That is, to the diagram, add a \emph{crossing circle}, which is a simple closed curve encircling the twist region and bounding a disk in $S^3$. We will always ensure that this new crossing circle introduces exactly four new crossings. So, the crossing circle is divided into four arcs. We ensure that two of these arcs are parallel in the diagram, and are as close as possible to the twist region. More precisely, we ensure that there is a square-shaped region of the diagram that includes two of the arcs of the crossing circle, and that there is a triangular region of the diagram that includes a crossing of the twist region and an arc of the associated crossing circle.

Let $L$ be the link consisting of $K$ along with all such crossing circles.  By work of Adams \cite{adams:aug}, the complement of $L$ is hyperbolic.  For $C$ a crossing circle, note that $S^3\setminus C$ is a solid torus.  Hence $S^3 \setminus L$ is homeomorphic to $S^3\setminus \tilde{L}$ where $\tilde{L}$ is obtained from $L$ by removing any even number of crossings from the twist region encircled by $C$, and the homeomorphism is given by twisting the solid torus $S^3\setminus C$.  

\begin{define}\label{def:L0-L2}
Define $L_2$ to be the diagram in which all except one or two crossings have been removed from each twist region of $L$ encircled by a crossing circle, depending on whether the number of crossings in that twist region in $L$ is odd or even, respectively.  When two crossings  remain in a twist region, we place the crossing circle so that two of its arcs  run through the bigon in the twist region between these two crossings.  

Define $L_0$ to be the diagram in which all except one or zero crossings have been removed. When one crossing remains in a twist region, either in $L_0$ or $L_2$, we still require that the crossing and the crossing circle that encircles it form a triangle in the diagram.

Finally, we let $K_i$ be the (diagram of the) link given by removing the crossing circles from the diagram of $L_i$, where $i=0$ or $2$.
An example of $K$, $L$, $L_2$ and $K_2$ is shown in Figure \ref{fig:links-ex}. \end{define}

\begin{figure}
  \includegraphics{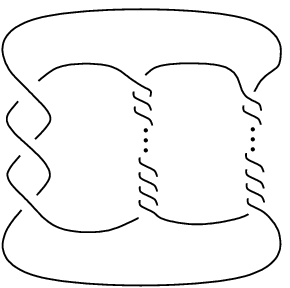}
  \hspace{.1in}
  \includegraphics{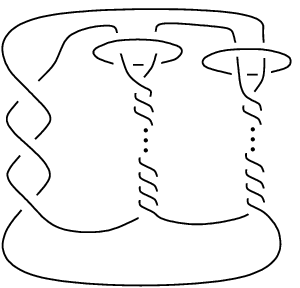}
  \hspace{.1in}
  \includegraphics{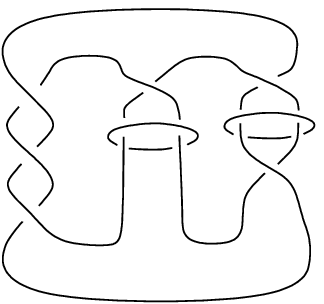}
  \hspace{.1in}
  \includegraphics{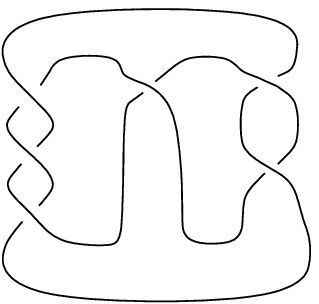}
  \caption{An example of link diagrams, left to right, $K$, $L$, $L_2$ and $K_2$.}
  \label{fig:links-ex}
\end{figure}

\begin{define}\label{def:crossing-assoc}
For any crossing encircled by a crossing circle of $L_i$, for $i=0$ or $2$, we say the crossing is \emph{associated} with the crossing circle.
\end{define}

We build twisted checkerboard surfaces as follows.  Note the diagram of $K_i$ is alternating.  Start with its checkerboard surfaces, colored blue and red.  Now, when we put the crossing circles of $L_i$ back into the diagram, a small regular neighborhood of each crossing circle intersects a checkerboard surface (either blue or red) in two meridian disks, shown 
on the left of Figure~\ref{fig:solidtorus}.

\begin{figure}[h]
  \input{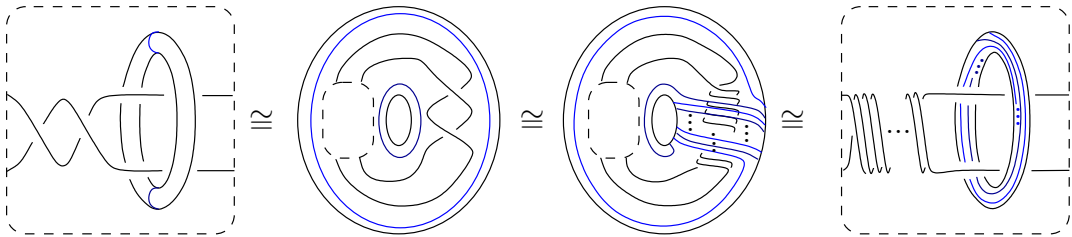}
  \caption{Effect of twisting on the intersection of the checkerboard surfaces of $K_i$ with neighborhood of a crossing circle.}
  \label{fig:solidtorus}
\end{figure}

Define the \emph{red and blue surfaces}, embedded in the exterior of $L_i$, to be the punctured red and blue checkerboard surfaces of $K_i$, respectively, punctured by the crossing circles of $L_i$.  Denote these by $R_i$ and $B_i$, where $i=0$ or $2$ depending on how many crossings are left in twist regions with an even number of crossings.

Consider what happens to $R_i$ and $B_i$ under the homeomorphism $(S^3\setminus L_i) \to (S^3 \setminus L)$.  In a neighborhood of the disk bounded by each crossing circle, the surface is twisted.  The two meridian curves go to $1/n$ curves, where $2|n|$ is the number of crossings removed to go from the twist region of $L$ to that of $L_i $. (The sign on $n$ must be chosen appropriately.) Let $R_\tw$ denote the minimal number of crossings removed from a twist region. Thus, $R_\tw$ is an even integer, and $2|n| \geq R_\tw$. Note also that crossings are only removed from a twist region if it has at least $N_\tw$ crossings, and hence
\begin{equation}\label{eqn:Rtw}
R_\tw \geq 2 \lfloor N_\tw /2 \rfloor \mbox{ if } i = 0 \mbox{ and }
R_\tw \geq 2 \lceil N_\tw /2 \rceil - 2 \mbox{ if } i = 2.
\end{equation}

To obtain $S^3 \setminus K$ from $S^3 \setminus L$, we do a meridian Dehn filling on each crossing circle.  To construct the twisted checkerboard surfaces, which we continue to color red and blue, do the following.  Each cross--sectional meridional disk of a crossing circle in $L$ intersects the punctured blue (or red) surface in $2|n|$ points on the boundary of the disk.  Connect opposite points on that disk by attaching an interval that runs through the center of the disk.  In other words, attach an $I$--bundle over $S^1$ that runs through the center of the Dehn--filling solid torus $|n|$ times.  See Figure \ref{fig:cross-sec}. If $n$ is odd, we attach an annulus.  If $n$ is even, each interval has both endpoints on the same curve, and so we attach two M\"obius bands.  In either case, the result is an immersed surface in $S^3\setminus K$ which we call the \emph{twisted checkerboard surface}.  We continue to color it red or blue, and we denote it by $S_{R,i}$ and $S_{B,i}$, respectively, where $i=0$ or $2$ depending on how many crossings we leave in twist regions containing an even number of crossings.

\begin{figure}
  \includegraphics{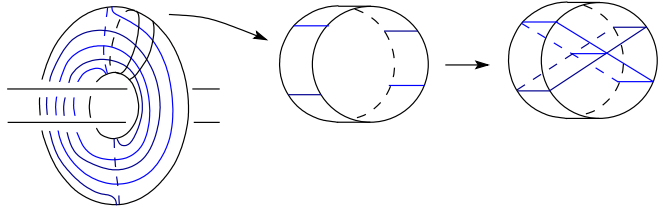}
  \caption{A cross section of the solid torus added to $S^3 \setminus L$, and how the surface extends into it.}
  \label{fig:cross-sec}
\end{figure}

The following is one of the main results of this paper.  

\begin{theorem}\label{thm:Bincompressible}
Let $f\co S_{B,i} \to S^3 \setminus K$ be the immersion of $S_{B,i}$ into $S^3\setminus K$.  Then this immersion is $\pi_1$--injective, provided $N_\tw \geq 54$ if $i=0$, and $N_\tw \geq 91$ if $i=2$.
\end{theorem}

Note that by switching the roles of the blue and red surfaces, Theorem \ref{thm:Bincompressible} also implies that the red surface $S_{R,i}$ is $\pi_1$--injective.  

There is also a version of this theorem which establishes that, in a suitable sense, the surfaces $S_{B,i}$ and $S_{R,i}$ are boundary--incompressible. Since this term is used in several distinct ways in the literature, we introduce an alternative. We say that a map $f\co S \to M$ between a surface $S$ and a 3--manifold $M$ satisfying $f(\partial S) \subset \partial M$ is \emph{boundary--$\pi_1$--injective} if, for any arc $\alpha \co I \to S$ with endpoints in $\partial S$, the existence of a homotopy (rel endpoints) of $f \circ \alpha$ into $\partial M$ implies the existence of a homotopy (rel endpoints) of $\alpha$ into $\partial S$.

\begin{theorem}\label{thm:bdryincompressible}
The surface $S_{B,i}$ is boundary--$\pi_1$--injective in $S^3 \setminus {\rm int}(N(K))$, provided $N_\tw \geq 54$ if $i=0$ and $N_\tw \geq 91$ if $i=2$. 
\end{theorem}

Here and throughout, $N(K)$ denotes an embedded regular neighborhood of $K$ in
$S^3$. 

\subsection{Acknowledgements}
Purcell is supported in part by NSF grant DMS--1252687, and by a Sloan Research Fellowship. The authors thank the referee for helpful comments on an earlier version of the paper, and for spotting an error in the original proof of Lemma \ref{lemma:adj-bigons}.

\section{Graphs from surfaces}\label{sec:graphs}
In the proofs of Theorems \ref{thm:Bincompressible} and \ref{thm:bdryincompressible}, the arguments for $i=0$ and $i=2$ are slightly different, but use much of the same machinery. In particular, both involve an analysis of a graph in a disk, obtained by the following lemmas.

\begin{lemma}\label{lemma:graph-bluevalence}
If $f\co S_{B,i} \to S^3\setminus K$ is not $\pi_1$--injective, then there is a map of a disk $\phi\co D \to S^3\setminus K$ with $\phi|_{\bdy D} = f \circ \ell$ for some essential loop $\ell$ in $S_{B,i}$, such that $\Gamma_B=\phi^{-1}(f(S_{B,i}))$ is a collection of embedded closed curves and an embedded graph in $D$.  The vertices of the graph are those points in  $D$ that map to a crossing circle. Each vertex in the interior of $D$ has valence a non-zero multiple of $2n_j$, where $2n_j$ is the number of crossings removed from the twist region at the relevant crossing circle. Each vertex on $\bdy D$ has  valence $n_j+1$.
\end{lemma}

\begin{proof}
If $f\co S_{B,i}\to S^3\setminus K$ is not $\pi_1$--injective, then there is some essential closed curve curve $\gamma$ in $S_{B,i}$ such that $[f(\gamma)]=0$ in $\pi_1(S^3\setminus K)$. Consider the curve $\gamma^2$.  Note we still have $[f(\gamma^2)]=0$ in $\pi_1(S^3\setminus K)$.  However, $\gamma^2$ lifts to the orientable double cover of $S_{B,i}$.  We will use $\gamma^2$ for this reason.

The fact that $[f(\gamma^2)]=0$ in the fundamental group gives a map $\phi\co D \to S^3\setminus K$ of a disk $D$ into $S^3\setminus K$ with $\phi|_{\bdy D} = f \circ \gamma^2$. However, we need to ensure that $\phi$ has the correct behaviour near $\partial D$, and so we construct $\phi$ in two stages, first near $\partial D$, and then over the remainder of the interior of $D$.

Let $C$ denote the crossing circles in $S^3\setminus K$, and let $N(C)$ denote a small regular neighborhood of $C$ in $S^3\setminus K$. By construction, $f^{-1}(C)$ is a collection of simple closed curves in $S_{B,i}$, with $f^{-1}(N(C))$ a collection of annuli and M\"obius bands.  We may ensure that the loop $\gamma^2$ in $S_{B,i}$ is transverse to these curves, and so it intersects the annuli and M\"obius bands transversely.  Each component of $\gamma^2 \cap f^{-1}(N(C))$ is therefore an arc which is mapped into $D^2 \times \{p\}$, for some $p \in S^1$, in a component of $N(C)$ and which runs from one prong of the relevant star to the opposite one.  We call this a \emph{sheet} of the star.

Now, $\gamma^2$ lifts to the orientable double cover of $S_{B,i}$. We may pick a consistent transverse orientation on this double cover, and using this, we may homotope $f \circ \gamma^2$ in this transverse direction. This homotopy is a map  of an annulus into $S^3 \setminus K$, which we take to be the restriction of $\phi$ to a collar neighbourhood of $D$. By carefully choosing this homotopy, we can ensure that the image of this homotopy has well--behaved intersection with $N(C)$, as follows.

Consider any arc component of $\gamma^2 \cap f^{-1}(N(C))$, mapping to $D^2 \times \{ p \}$ in a component of $N(C)$. Now note $\phi(\bdy D)$ runs through $\{ 0 \} \times \{ p \}$ in a prong of a star.  In a ball around $\{ 0 \} \times \{ p \}$, the image of $f(S_{B,i})$ is homeomorphic to the product of an open interval and that star.  The homotopy of $\phi(N(\bdy D))$  pushes it so that it remains on one side of this sheet of the star. As for an arc component of $\gamma^2 \setminus f^{-1}(N(C))$, homotoping a neighborhood of that arc in the direction of the transverse orientation ensures the neighborhood only intersects $f(S_{B,i})$ on that arc. So, in all cases, near $\partial D$, $\phi^{-1}(f(S_{B,i}))$ looks like a graph with each vertex on $\bdy D$ having valence $n_j + 1$. 

We have thus defined $\phi$ in a collar neighbourhood of $\partial D$. Since we are assuming that the $[f(\gamma^2)]$ is trivial in $\pi_1(S^3 \setminus K)$, we may extend this to a map $\phi \co D \rightarrow S^3 \setminus K$. Using a small homotopy supported away from a neighbourhood of $\partial D$, make $\phi$ transverse to all crossing circles, and transverse to $f(S_{B,i})$.  Consider $\Gamma_B =\phi^{-1}(f(S_{B,i}))$ on $D$. Because $S_{B,i}$ is embedded in $S^3\setminus K$ except at crossing circles, $\Gamma_B$ consists of embedded closed curves, and embedded arcs (edges) with endpoints corresponding to points of intersection of crossing circles (vertices). Each vertex in the interior of $D$ corresponds to the intersection of $\phi(D)$ with a crossing circle in $S^3\setminus K$, and the intersection is transverse.  Hence it meets the boundary of a neighborhood of that crossing circle in a meridian.  The blue surface $f(S_{B,i})$ meets the boundary of the neighborhood of that crossing circle in two curves of slope $\pm 1/n_j$, where $2n_j$ is the number of crossings in $K$ removed from the twist region of that crossing circle.  Hence the valence of the vertex in the interior of $D$ is $2 n_j$.
\end{proof}

Note that the vertices in the interior of $D$ have valence precisely $2n_j$ in the above construction. However, later, it will be convenient to permit the existence of vertices that have valence a non-zero multiple of $2n_j$.

A similar result holds when $f$ is not boundary--$\pi_1$--injective.

\begin{lemma}\label{lemma:graph-bdyincompr}
Suppose $f\co S_{B,i} \to S^3\setminus {\rm int}(N(K))$ is not boundary--$\pi_1$--injective. Then there is a map of a disk $\phi\co D \to S^3\setminus {\rm int}(N(K))$ with $\partial D$ expressed as a concatenation of two arcs, one mapped by $\phi$ into $\partial N(K)$, and the other factoring through an essential arc in $S_{B,i}$. Moreover, $\Gamma_B =\phi^{-1}(f(S_{B,i}))$ is a collection of embedded closed curves and an embedded graph on $D$ whose edges have endpoints either at vertices where $\phi(D)$ meets a crossing circle, or on $\phi^{-1}(\partial N(K))$ on $\bdy D$.  Each vertex in the interior of $D$ has valence a non-zero multiple of $2n_j$, where $2n_j$ is the number of crossings removed from the twist region at the relevant crossing circle. Each vertex in the interior of the arc in $\bdy D$ that maps to $S_{B,i}$ has valence $n_j+1$. Each vertex on the arc in $\bdy D$ that maps to $\partial N(K)$ has valence one.
\end{lemma}

\begin{proof}
If $f\co S_{B,i} \to S^3\setminus {\rm int}(N(K))$ is not boundary--$\pi_1$--injective, then there is a nontrivial arc on $S_{B,i}$ which is homotopic (rel endpoints) into $\partial N(K)$ in $S^3\setminus {\rm int}(N(K))$.  This gives us a map of a disk $\phi\co D \to S^3\setminus {\rm int}(N(K))$ with $\bdy D$ consisting of the two arcs required by the lemma. Let $\alpha$ be the sub-arc in $\partial D$ that maps via the essential arc in $S_{B,i}$.  Again, we need to control $\phi$ near $\partial D$, and so we construct $\phi$ in two stages. The arc $\alpha$ lifts to the orientable double cover of $S_{B,i}$, which is transversely orientable, and by pushing $\alpha$ in this transverse direction, we obtain the map $\phi$ in a neighbourhood of $\alpha$. Similarly, using the fact that $\partial N(K)$ is transversely orientable, we can extend the definition of $\phi$ over a collar neighbourhood of $\partial D$. Now extend $\phi$ over all of $D$, and then make it transverse to to all crossing circles, and transverse to $f(S_{B,i})$. Let $\Gamma_B=\phi^{-1}(f(S_{B,i}))$ on $D$.  Because $S_{B,i}$ is embedded in $S^3\setminus K$ except at crossing circles, $\Gamma_B$ consists of embedded closed curves, embedded arcs (edges) with endpoints corresponding to points of intersection of crossing circles (vertices), or with endpoints on $\partial N(K)$.

As in the proof of Lemma~\ref{lemma:graph-bluevalence}, a vertex in the interior of $D$ corresponds to a transverse intersection of $\phi(D)$ with a crossing circle in $S^3\setminus K$.  Hence the vertex has valence $2n_j$. Near a vertex in the interior of the arc in $\bdy D$ that maps to $S_{B,i}$, the graph looks like half a meridian disk for a crossing circle, and so has valence $n_j+1$. At a vertex on the arc in $\bdy D$ that maps to $\partial N(K)$, the arc in $\bdy D$ is transverse to $S_{B,i}$, and so this vertex of $\Gamma_B$ has valence one.
\end{proof}

The following well--known result will be central to our proof.

\begin{lemma}\label{lemma:graph-in-2-sphere}
Let $\Gamma$ be a connected graph in the 2-sphere that has no bigons and no monogons, and that is neither an isolated vertex nor a single edge joining two vertices. Then $\Gamma$ contains at least three vertices with valence less than $6$.
\end{lemma}

\begin{proof} We may add edges to the graph until every complementary region is triangular. Let $V$, $E$ and $F$ denote the number of vertices, edges and faces. Then $2E = 3F$. Hence,
$$2 = V - E + F = V - E/3 = \sum_v (1 - (d(v)/6)),$$
where the sum runs over each vertex $v$, and $d(v)$ denotes the valence of a vertex. Since $d(v) > 0$ for each $v$, we deduce that $1 - (d(v)/6) < 1$, and hence there must be at least three vertices with valence less than $6$.
\end{proof}

\begin{lemma}\label{lemma:planar-graph}
Let $\Gamma$ be a connected graph in the disk $D$ that includes $\partial D$, that contains no bigons and no monogons. Then either there is some vertex in the interior of $D$ with valence at most $5$, or there are at least three vertices on the boundary with valence at most $3$.
\end{lemma}

\begin{proof}
Double the disk $D$ to form a 2--sphere, and double $\Gamma$ to form a graph $\Gamma^+$ in this 2--sphere. Now $\Gamma^+$ contains no bigons or monogons, since this was true of $\Gamma$. By Lemma \ref{lemma:graph-in-2-sphere}, $\Gamma^+$ must have at least three vertices with valence less than $6$. If one of these vertices is disjoint from the copy of $\partial D$ in the 2--sphere, the lemma is proved. On the other hand, if all three vertices lie on $\partial D$, then their valence in $\Gamma$ is at most $3$.
\end{proof}

\begin{lemma}\label{lemma:planar-graph2}
Let $\Gamma$ be a connected graph on a disk that includes the boundary of the disk and contains no monogons. Suppose each interior vertex of $\Gamma$ has valence at least $R_\tw$ and each boundary vertex has valence at least $(R_\tw/2)+1$, with at most two exceptions. Then $\Gamma$ must have more than $(R_\tw/6)-1$ adjacent bigons. 
\end{lemma}

\begin{proof}
Suppose that the lemma is not true. Then every collection of adjacent bigons has at most $(R_\tw/6)$ edges. 
Collapse each family of adjacent bigons to a single edge, forming a graph $\overline{\Gamma}$. By Lemma~\ref{lemma:planar-graph},
$\overline{\Gamma}$ contains a vertex in the interior of the disk with valence at most $5$ or at least three vertices on the boundary with 
valence at most $3$. In the former case, the vertex came from a vertex of $\Gamma$ with valence at most
$5(R_\tw/6)$, which is less than $R_\tw$. In the latter case, each vertex came from a vertex of $\Gamma$ with
valence at most $3(R_\tw/6)$, which is less than $(R_\tw/2)+1$. In both cases, we get a contradiction.
\end{proof}

We now focus on the graph $\Gamma_B$ provided by Lemma~\ref{lemma:graph-bluevalence} or Lemma~\ref{lemma:graph-bdyincompr}.

We declare certain edges of $\Gamma_B$ with at least one endpoint on $\bdy D$ to be \emph{trivial}.  The precise condition will be given in
Definition~\ref{def:trivial-edge}, but it does not concern us here. However, if there is a bigon region of $\Gamma_B$, then either both of its edges are trivial or neither are. (See Lemma~\ref{all-but-one-edge-trivial}.) As a result, we say that a \emph{trivial bigon family} is a connected union of trivial bigons, homeomorphic to a disk, and which is maximal, in the sense that none of the bigons are incident along an edge to a trivial bigon not in the family. In addition, when two edges of a triangular region of $\Gamma_B$ are trivial, then so is the third, again by Lemma~\ref{all-but-one-edge-trivial}. Another property of trivial edges is that the vertices at their endpoints correspond to the same crossing circle.

\begin{lemma}\label{lemma:adj-bigons}
Let $\Gamma_B$ be the graph in $D$ provided by Lemma~\ref{lemma:graph-bluevalence}. Suppose that $\Gamma_B$ has no monogons. Assume also that there are no trivial edges of $\Gamma_B$ in $\partial D$. Then $\Gamma_B$ must have more than $(R_{\tw}/18)-1$ adjacent non-trivial bigons, where $R_\tw$ is the minimal number of crossings removed from a twist region.
\end{lemma}

\begin{proof}
Note that it follows immediately from Lemma \ref{lemma:planar-graph2} that $\Gamma_B$ must have more than $(R_\tw/6)-1$ adjacent bigons (which may possibly be trivial). Thus, the main challenge in the proof is to deal with trivial bigons.
To do so, we will form a new graph, closely related to $\Gamma_B$, with all trivial bigons removed. We form the graph in several steps, in order to keep track of its properties, including the valence of its vertices and the nature of its bigons.

\vspace{.1in}

\underline{Step 1: Restrict to a subgraph in a subdisk.} In this step, we will focus on a subdisk $D'$ of $D$ (which may be all of $D$). We will also focus on
$\Gamma_B' = \Gamma_B \cap D'$, which will be a subgraph of $\Gamma_B$. This will have various properties which we will enumerate below, including the following:

\begin{enumerate}
\item[(1)] $\partial D' \subset \Gamma_B'$.
\item[(2)] The interior of $D'$ has non-empty intersection with $\Gamma_B'$.
\item[(3)] $\Gamma_B'$ is connected.
\item[(4)] $\Gamma_B'$ contains no edge loops, which are edges that start and end at the same vertex, except possibly one lying in $\partial D'$.
\end{enumerate}

We will also declare that certain vertices on the boundary of $D'$ are \emph{exceptional}. We will have the following property:
\begin{enumerate}
\item[(5)] Any unexceptional vertex of $\Gamma_B'$ is equal to a vertex of $\Gamma_B$ with the same valence.
\end{enumerate}

Initially, we set $D' = D$ and $\Gamma_B' = \Gamma_B$ and we have no exceptional vertices.
This disk and graph satisfy (1) and (5). We may also assume that they satisfy (2), as otherwise the lemma holds trivially.

If the graph $\Gamma_B$ is not connected, then there must be an innermost component. Because $\partial D$ is connected, that innermost component contains no edges meeting $\partial D$. In particular, it contains no trivial bigon families. Then Lemma~\ref{lemma:planar-graph2} immediately implies the result. 
Thus, we may now assume that $\Gamma_B'$ is connected, and so satisfies (1), (2) and (3).

The graph $\Gamma_B'$ may contain an edge loop. Because $\Gamma_B$ has no monogons, an edge loop must enclose parts of $\Gamma_B$ in its interior in $D$. If $\Gamma_B'$ contains an edge loop, pick an innermost one and replace $D'$ by the disk bounded by that edge loop, and replace $\Gamma_B'$ by the portion of $\Gamma_B$ in that subdisk. In this case, declare that the vertex in the boundary of $D'$ is exceptional. Thus, we obtain property (4).

The next property that we wish to ensure is:
\begin{enumerate}
\item[(6)] If two trivial bigon families share an interior vertex, then their vertices on $\partial D'$ are distinct.
\end{enumerate}

Suppose we had two trivial edges with same endpoints, one of which is a vertex in the interior of $D'$, but that are not part of the same trivial bigon family. 
These edges bound a subdisk. In this case, we set $D'$ to be an innermost
such subdisk. Since the edges were not part of a trivial bigon family, (2) continues to hold.
In this case, we declare that the two vertices in the boundary of $D'$ are exceptional.

At this stage, if $D' \not= D$, 
the subdisk $D'$ and graph $\Gamma_B'$ satisfy properties (1) - (6), as well as two additional properties (7) and (8) below, and no further work is needed. So in this case, we finish Step 1, and pass to Step 2. 

So, we now assume that $D' = D$, and that (1) - (6) hold.

Next, the graph $\Gamma_B$ may contain \emph{trivial stars}, which are defined to be a collection of at least two trivial bigon families that are all incident to the same interior vertex. By (6), the trivial bigon families that are part of a trivial star intersect $\partial D$ in distinct vertices. See Figure~\ref{fig:trivialbigonsfig}. The graph $\Gamma_B$ may also contain trivial edges with both endpoints on $\partial D$. Because $\Gamma_B$ has no trivial edges in $\partial D$, each such edge splits $D$ into two disks. 

\begin{figure}
 \includegraphics[width=3in]{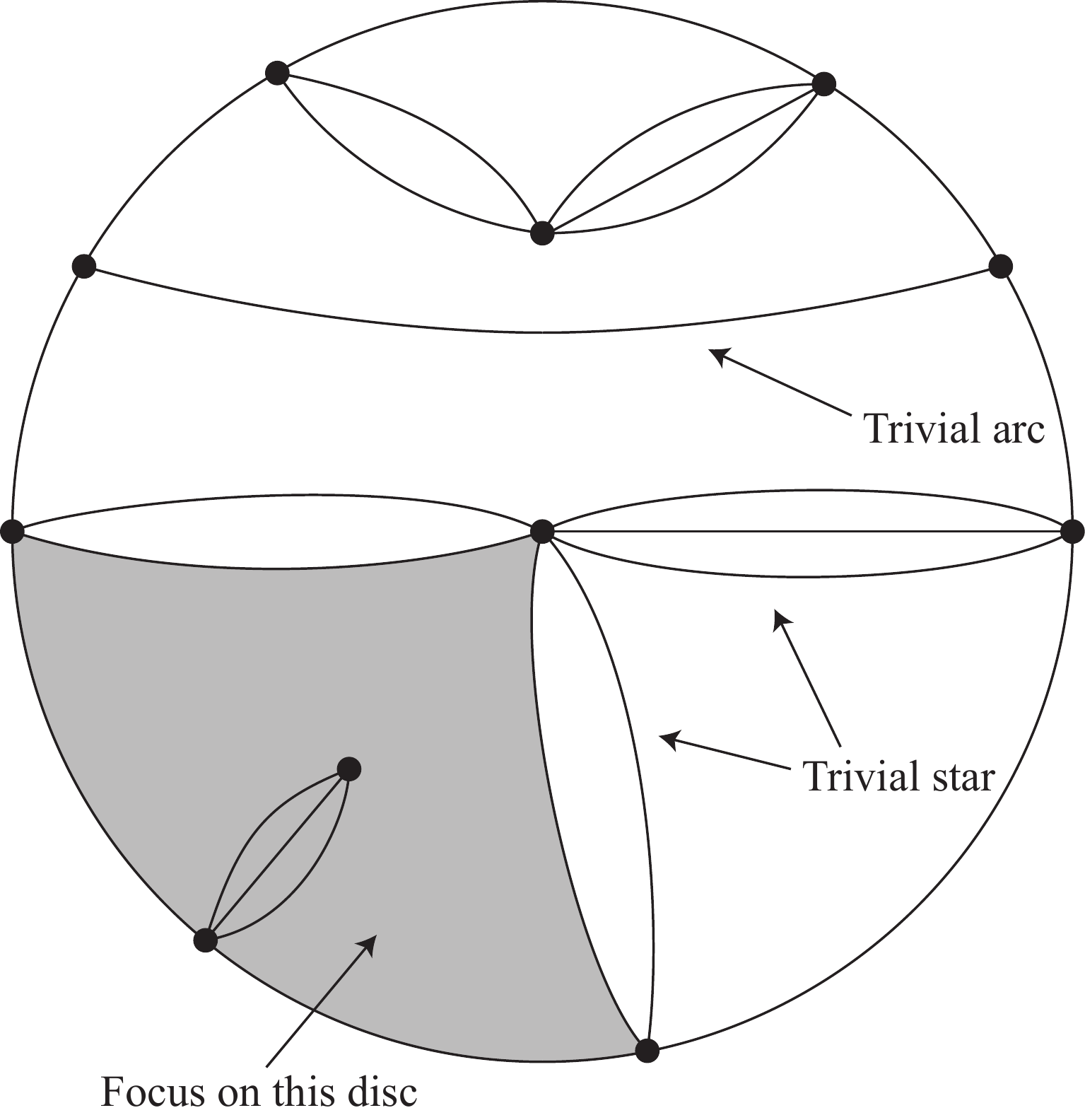}
\caption{Trivial arcs and trivial stars in $\Gamma_B$}
\label{fig:trivialbigonsfig}
\end{figure}

By (4) and (6), the trivial stars and trivial edges with both endpoints on $\partial D$ separate $D$ into subdisks, as in Figure~\ref{fig:trivialbigonsfig}. As well as properties (1), (3), (4) and (6) above, these subdisks have the following properties. 
\begin{enumerate}
\item[(7)] In each subdisk, each interior vertex meets at most one trivial bigon family. Moreover, all remaining trivial bigon families have one boundary vertex, and one (distinct) interior vertex.
\end{enumerate}

Pick a subdisk $D'$ that is outermost in $\partial D$. This satisfies (2), because otherwise the disk is a triangle or bigon and its intersection with $\partial D$ was a trivial edge (see Lemma~\ref{all-but-one-edge-trivial}), which contradicts our assumption. When $D'$ is a subdisk separated off by a trivial star, we term the three vertices that lay in the trivial star exceptional. When $D'$ is a subdisk separated off by a trivial arc, we term the two vertices at the endpoint of this arc exceptional.

We note that, in all cases, $D'$ satisfies (1) - (7), as well as the following:
\begin{enumerate}
\item[(8)] The vertices of $\Gamma_B'$ in $D'$ have valence as follows.
  \begin{itemize}
  \item Interior vertices have valence $2n_j \geq R_\tw$, where $2n_j$ is the number of crossings removed from the twist region at the relevant crossing circle.
  \item Vertices on the boundary have valence $n_j+1 \geq R_\tw/2+1$, except possibly the exceptional ones. 
  \end{itemize}
\end{enumerate}

\vspace{0.1in}
\underline{Step 2: Remove remaining trivial bigon families.} By property (7) above, the only remaining trivial bigon families have one vertex on the boundary, and the other on an interior vertex that meets no other trivial bigon families. 

We now form a graph $\Gamma$ in $D'$ as follows. Consider each unexceptional vertex $v$ on $\partial D'$ in turn. If $v$ is incident to no trivial bigons, then we leave it untouched. If $v$ is incident to $k \geq 1$ trivial bigon families, then the other endpoints of these bigon families are distinct vertices in the interior of $D'$. Replace $v$ with $k$ vertices on $\partial D'$. We view these $k$ vertices as the vertices at the endpoints of the bigon families. Remove the edges meeting $v$ that were part of a trivial bigon family. The edges of $\Gamma_B'$ that used to end at one of these bigon family vertices now end at one of the new vertices. There may have been some edges of $\Gamma_B'$ that ended on $v$ but that were not part of a trivial bigon family. A new endpoint has to be found for such edges. There are two possible choices: we make one arbitrarily. See Figure \ref{fig:trivialbigonsfig2}. Finally, if $\Gamma_B'$ had three exceptional vertices, collapse two of these to a single vertex.

\begin{figure}
\input{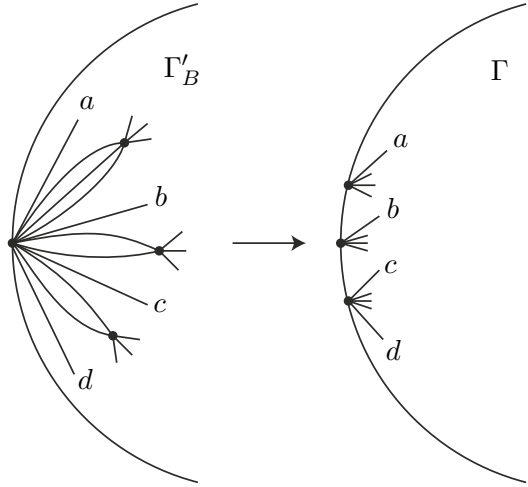}
\caption{Construction of $\Gamma$}
\label{fig:trivialbigonsfig2}
\end{figure}

The resulting graph $\Gamma$ has the following properties.

\begin{enumerate}
\item[(a)] It has no monogons. This is because any monogon region of $\Gamma$ must have come from a monogon of $\Gamma_B'$, or an edge with endpoints on two vertices that were collapsed to one, hence coming from a bigon region in $\Gamma_B'$. There are no monogons in $\Gamma_B'$ (by properties (2) and (4)). An edge forming a bigon with an edge collapsed to one vertex must have been a trivial edge, by Lemma~\ref{all-but-one-edge-trivial}, since in all cases the collapsed edge was trivial. But then the other edge would have been part of the corresponding bigon family, by maximality of families. So there are no monogons. 

\item[(b)] Bigons in $\Gamma$ come from those in $\Gamma_B$, with the following exceptions. If $\Gamma_B$ has a triangular region with one of its edges on a trivial bigon family that is collapsed to form $\Gamma$, then that triangular region becomes a new bigon in $\Gamma$. Similarly, if $\Gamma_B$ has a square region with two of its edges being part of
trivial bigon families, then this may collapse to form a bigon.

We deduce that any collection of adjacent bigons in $\Gamma$ came from collections of adjacent non-trivial bigons of $\Gamma_B$, plus possibly triangular or square regions. We claim that there can be no more than two triangular regions giving rise to a set of adjacent non-trivial bigons of $\Gamma_B$, and no more than one square region, and thus adjacent bigons in $\Gamma$ come from no more than three collections of non-trivial bigons of $\Gamma_B$, as follows. Distinct bigons in $\Gamma_B$ become adjacent in $\Gamma$ only if their vertices are separated by a trivial bigon family that is collapsed. There are two vertices at the endpoints of the bigons in $\Gamma$. These can come from collapsing at most two trivial bigon families, by property (7) and our choice of collapsing just one edge of a trivial bigon star. Thus one square, or one or two triangles, are possible, but this will group together at most three non-trivial bigon collections in $\Gamma_B$. 

\item[(c)] Valences of vertices of $\Gamma$ are as follows, where again $2n_j$ denotes the number of crossings removed from the relevant crossing circle in each case. 
  \begin{itemize}
  \item Any interior vertex has valence $2n_j\geq R_\tw$: all remaining interior vertices of $\Gamma$ came from interior vertices of $\Gamma_B$ that met no trivial bigons.

  \item A non-exceptional boundary vertex that came from an original boundary vertex of $\Gamma_B$ has valence $n_j+1 \geq R_\tw/2+1$, since it was not affected by the modification to $\Gamma$.

  \item A non-exceptional boundary vertex that came from a trivial bigon family has valence at least $n_j+1 \geq R_\tw/2+1$. This is because one vertex of the trivial bigon family lies on $\partial D$ in $\Gamma_B$, and so at most $n_j-1$ edges of $\Gamma_B$ lie in that trivial bigon family. At the other endpoint of the bigons is a vertex of $\Gamma_B$ in the interior of $D$. It therefore has valence $2n_j$, for the same $n_j$, using the fact that trivial edges have endpoints on vertices corresponding to the same crossing circle. So, the corresponding vertex of $\Gamma$ has valence at least $n_j+1$.

  \item Exceptional vertices may have lower valence. However, by construction, $\Gamma$ has at most two of these.
  \end{itemize}
\end{enumerate}

Thus, the hypotheses of Lemma~\ref{lemma:planar-graph2} apply to $\Gamma$. We therefore deduce that it has a collection of more than $(R_\tw/6)-1$ adjacent bigons. By property~(b) above, all but at most two of these came from a non-trivial bigon of $\Gamma_B$. These are divided into at most three collections of adjacent non-trivial bigons of $\Gamma_B$. So we deduce that $\Gamma_B$ has more than $(R_\tw/18)-1$ adjacent non-trivial bigons.
\end{proof}

\begin{lemma}\label{lemma:adj-bigons-2}
Let $\Gamma_B$ be the graph in $D$ provided by Lemma~\ref{lemma:graph-bdyincompr}. Suppose that $\Gamma_B$ has no monogons. Assume also that there are no trivial edges of $\Gamma_B$ in $\partial D$. Then $\Gamma_B$ must have more than $(R_\tw/18)-1$ adjacent non-trivial bigons or there are more than $(R_\tw/18)-1$ adjacent triangles, where one edge of each triangle lies in $\phi^{-1}(\partial N(K))$.
\end{lemma}

\begin{proof} 
We argue as in the proof of Lemma~\ref{lemma:adj-bigons}. Whenever we pass to a subdisk we make sure that it does not contain the arc $\phi^{-1}(\partial N(K))$, and the argument then proceeds exactly as in Lemma \ref{lemma:adj-bigons}. So, suppose that we do not pass to a subdisk. Then we double $D^2$ along the arc $\phi^{-1}(\partial N(K))$. The two copies of $\Gamma_B$ become a single graph, and the vertices that lay on $\phi^{-1}(\partial N(K))$ become the midpoints of edges.
In the resulting graph, every vertex on the boundary of the disk has valence $n_j+1$, where $2n_j$ is the number of crossings removed from the twist region at the relevant crossing circle. We can therefore apply the argument of Lemma \ref{lemma:adj-bigons}. In all cases, the conclusion is that there are more than
$(R_\tw/18)-1$ adjacent non-trivial bigons. In the original disk $D$, these give more than $(R_\tw/18)-1$ adjacent non-trivial bigons, or more than $(R_\tw/18)-1$ adjacent triangles, where one edge of each triangle lies in $\phi^{-1}(\partial N(K))$.
\end{proof}

By analyzing properties of alternating diagrams, we will show below that the graph on $D$ coming from $S_{B,0}$ cannot contain three adjacent non-trivial bigons.  The graph from $S_{B,2}$ cannot contain five adjacent non-trivial bigons.  Assuming these results, we give the proof of Theorem~\ref{thm:Bincompressible}.

\begin{proof}[Proof of Theorem \ref{thm:Bincompressible}]
If $f\co S_{B,i} \to S^3\setminus K$ is not $\pi_1$--injective, then Lemma~\ref{lemma:graph-bluevalence} implies there is a map of a disk $\phi\co D \to S^3\setminus K$ with $\phi|_{\bdy D} = f \circ \ell$ for some essential loop $\ell$ in $S_{B,i}$, such that $\Gamma_B=\phi^{-1}(f(S_{B,i}))$ is a collection of embedded closed curves and an embedded graph in $D$.  Each vertex in the interior of $D$ has valence a non-zero multiple of $2n_j$, where $2n_j$ is the number of crossings removed from the twist region at the relevant crossing circle. Each vertex on $\bdy D$ has  valence $n_j+1$. In Section~\ref{sec:surfaces}, we will define a measure of complexity for such maps $\phi$, and we choose $\phi$ to have minimal complexity.

Lemma~\ref{lemma:scc} implies that there are no simple closed curves in $\Gamma_B$.  Lemma~\ref{lemma:no-monogons} implies that there are no monogons.
Lemma~\ref{lemma:no-boundary-trivial-arcs} states that there are no edges in $\partial D$ that are trivial. So, by Lemma~\ref{lemma:adj-bigons}, $\Gamma_B$ has more than $(R_\tw/18)-1$ adjacent non-trivial bigons. On the other hand, Proposition~\ref{prop:L0-no-3adjbigons} implies there cannot be three adjacent non-trivial bigons when $i=0$.  If $N_\tw \geq 54$, then by equation \eqref{eqn:Rtw}, $R_\tw \geq 54$, and there are more than two adjacent bigons.  So we deduce that the surface $S_{B,0}$ is $\pi_1$--injective in this case.

For the case $i=2$, if $N_\tw$ is at least $91$, equation \eqref{eqn:Rtw} implies that $R_\tw$ is at least $90$, and Lemma~\ref{lemma:adj-bigons} implies that the graph on $D$ contains more than four adjacent non-trivial bigons.
But now, Proposition~\ref{prop:L2-no-5adjbigons} implies that there cannot be five adjacent non-trivial bigons on $D$ when $i=2$.
\end{proof}

Similarly, assuming the above results as well as Lemma~\ref{lemma:no-3BBKtriangles}, we may prove Theorem~\ref{thm:bdryincompressible}.

\begin{proof}[Proof of Theorem \ref{thm:bdryincompressible}]
Suppose $S_{B,i}$ is not boundary--$\pi_1$--injective.  Then Lemma~\ref{lemma:graph-bdyincompr} gives a graph on a disk $D$. This satisfies the hypotheses of Lemma~\ref{lemma:adj-bigons-2}. So, $\Gamma_B$ must have at least $(R_\tw/18)-1$ adjacent non-trivial bigons or there at least 
that many adjacent triangles, where one edge of each triangle lies in $\phi^{-1}(\partial N(K))$. The former case is ruled out by Proposition~\ref{prop:L0-no-3adjbigons} or Proposition~\ref{prop:L2-no-5adjbigons}. The latter case is ruled out by Lemma~\ref{lemma:no-3BBKtriangles}.
\end{proof}

\section{Diagrams and properties}
Our goal is to complete the proofs of the Lemmas referenced in the proofs of Theorems \ref{thm:Bincompressible} and \ref{thm:bdryincompressible}.  The arguments are combinatorial, relying on properties of the diagrams of the links defined in Definition \ref{def:L0-L2} above.  In this section we discuss these diagram properties.

To simplify the argument, we will consider a modification of the diagrams of $L$, $L_i$, and $K_i$.  In particular, recall that we obtained $L$ from $K$ by adding a crossing circle to each twist region of $K$ that had more than $N_\tw$ crossings.  These crossing circles either meet the blue or the red checkerboard surfaces of $K$.  Form a new augmented alternating link $L_B$ from $K$ by only adding those crossing circles that meet the blue checkerboard surface of $K$. Obtain $L_{B,i}$ by removing pairs of crossings encircled by each crossing circle of $L_B$, leaving either one or $i$ crossings, where $i=0$ or $i=2$.  Just as in Definition~\ref{def:L0-L2}, we place the crossing circle in the diagram so that any associated crossing forms a triangle with the crossing circle.  In particular, if $i=2$ and a crossing circle is associated with two crossings, the crossing circle runs through the bigon formed by those two crossings.

Finally, obtain $K_{B,i}$ by removing crossing circles of $L_{B,i}$.  Note $K_{B,i}$ differs from $K$ in that red bigon regions have been removed.

Now, notice that the surface $B_i$ embedded in $S^3\setminus L_i$ can also be embedded in $S^3\setminus L_{B,i}$.  However, now the red surface $R_i$ in $S^3\setminus L_{B,i}$ is homeomorphic to the red checkerboard surface of $K$, and to the red checkerboard surface of $K_{B,i}$.  

The following lemma discusses the primality of the diagrams of $K_{B,0}$ and $K_{B,2}$.

\begin{lemma}\label{lemma:K2-0-prime}
The diagrams of $K_{B,2}$ and $K_2$ are prime.

The diagram of $K_{B,0}$ may not be prime.  However, if $\gamma$ is a simple closed curve giving a counterexample to primality of $K_{B,0}$, then $\gamma$ consists of two arcs, one $\gamma_R$ in the red surface, and one $\gamma_B$ in the blue.  The arc $\gamma_R$ can be homotoped (rel endpoints) to run transversely through some crossing disk in the diagram of $L_{B,0}$, intersecting it exactly once.  The arc $\gamma_B$ can be homotoped (rel endpoints) to be disjoint from crossing disks.

Finally, if the diagram of $K_0$ is not prime, then a simple closed curve giving a counterexample to primality of $K_0$ must link a crossing circle of $L_0$.
\end{lemma}

Here, a \emph{crossing disk} denotes the twice punctured disk with boundary on the crossing circle, embedded transverse to the plane of projection of the diagram of $L_i$ (or $L$).  There is one crossing disk for each crossing circle.  The collection of all disks is embedded.

\begin{proof}
Suppose $\gamma$ is a closed curve in the projection plane meeting the diagram of $K_{B,2}$ twice.  Then we claim $\gamma$ can be isotoped to be disjoint from any crossing disk of $L_{B,2}$.  For if $\gamma$ meets such a disk in the blue surface, since crossing disks intersect blue regions in simple arcs running from the boundary to a point in the interior of the region, the curve $\gamma$ can simply be pulled off the end of the crossing disk.  If $\gamma$ meets a crossing disk $D$ in the red surface, then it can either be pulled off the disk, or it meets an edge of $L_{B,2}$ running between a crossing associated with $D$ and the crossing disk.  Now slide $\gamma$ along this edge towards $D$, in a neighborhood of the edge, until $\gamma$ no longer meets $D$ in the red surface.  Then slide $\gamma$ off of $D$ in the blue surface, as above.  Thus in all cases, we may isotope $\gamma$ to be disjoint from $D$.  Now put back all the crossings to obtain the diagram of $K$.  Because $\gamma$ misses all crossing disks, it misses all these crossings, hence gives a curve in $K$ meeting the diagram twice.  Because the diagram of $K$ is prime, there are no crossings on one side of the curve.  Then the same is true for $K_{B,2}$: there are no crossings in the diagram on one side of $\gamma$, and $K_{B,2}$ is prime.

Obtain $K_2$ from $K_{B,2}$ by removing blue bigons in the diagram of $K_{B,2}$.  Since $K_{B,2}$ is prime, the same argument as above with $K_{B,2}$ replacing $K$, and $K_2$ replacing $K_{B,2}$, and the red surface replacing blue implies that the diagram of $K_2$ is prime.

Now consider $K_{B,0}$.  Let $\gamma$ be a curve meeting the diagram twice with crossings on either side.  Because $\gamma$ meets the diagram twice in interiors of edges of the diagrams, and because each edge of the diagram meets the red and blue surfaces on either side, we obtain the claim that $\gamma$ consists of two arcs, $\gamma_R$ in the red surface and $\gamma_B$ in the blue.  Note that $\gamma_R$ lies in a single (red) region of the diagram graph, and such a region is a disk.  Similarly for $\gamma_B$.  As above, we may homotope $\gamma_B$ (rel endpoints) to avoid all crossing disks.

Now, the diagrams of $K_{B,0}$ differs from that of $K$ only in that an even number of crossings in select twist regions have been removed. Homotope $\gamma_R$ (rel endpoints) in the red region to meet as few crossing disks as possible.  If $\gamma_R$ does not meet a crossing disk, then neither does $\gamma$, and we may put back crossings to obtain $K$ without increasing the number of intersections of $\gamma$ with the diagram.  But then $\gamma$ gives an embedded closed curve in the diagram of $K$ meeting the diagram twice with crossings on either side, contradicting primality of $K$.

So $\gamma_R$ must meet a crossing disk.  Because $\gamma_R$ lies in a single red region (which is a disk), $\gamma_R$ can be homotoped (rel endpoints) to meet the crossing disk exactly once.  Because $\gamma_B$ does not meet the crossing disk, $\gamma$ must link the corresponding crossing circle exactly once.

Finally, consider $K_0$.  If $\gamma$ is a simple closed curve meeting the diagram of $K_0$ twice with crossings on either side, modify the diagram by putting back crossings bounding blue bigons to obtain $K_{B,0}$.  If after this modification $\gamma$ still meets the diagram twice with crossings on either side, then the previous paragraph implies that $\gamma$ links a crossing circle of $L_0$.  If not, then $\gamma$ in $K_{B,0}$ must run through a sequence of blue bigons, consisting of crossings in a twist region.  As before, $\gamma$ must link the associated crossing circle.
\end{proof}

Lemma \ref{lemma:K2-0-prime} has the following important consequence.

\begin{lemma}\label{lemma:blue-prime}
For $i=1, 2$, label the regions of the complement of the diagram of $K_{B,i}$ blue or red depending on whether they meet the blue or red surface.  Note each crossing circle of $L_{B,i}$ intersects two blue regions.
\begin{enumerate}
\item\label{lmitm:distinct-regions-cross} The blue regions on opposite sides of a crossing of $L_{B,i}$ cannot agree, for $i=0, 2$.
\item\label{lmitm:distinct-red} The red regions on opposite sides of a crossing of $L_{B,2}$ cannot agree.  (Note this is not necessarily true for $L_{B,0}$.)
\item\label{lmitm:distinct-regions} The two blue regions that meet a single crossing circle of $L_{B,i}$ cannot agree.  That is, each crossing circle of $L_{B,i}$ meets two distinct blue regions, for $i=0, 2$.
\item\label{lmitm:assoc-cross} Suppose the distinct blue regions meeting a single crossing circle meet at the same crossing of the diagram of $K_{B,i}$.  Then that crossing is associated with the crossing circle, as in Definition \ref{def:crossing-assoc}, for $i=0, 2$.
\end{enumerate}
\end{lemma}

\begin{proof}
For \eqref{lmitm:distinct-regions-cross}, if the blue regions on opposite sides of the same crossing do agree, then draw an arc in this region from one side of the crossing to the other.  Close this into a simple closed curve by drawing a short arc in a red region meeting that crossing, close enough to the crossing that it does not meet any crossing disks.  This simple closed curve contradicts Lemma \ref{lemma:K2-0-prime}.  Similarly for item \eqref{lmitm:distinct-red}, if red regions on opposite sides of a crossing of $L_{B,2}$ agree, then we may connect them into a closed curve with crossings on either side, contradicting Lemma~\ref{lemma:K2-0-prime}.

For item \eqref{lmitm:distinct-regions}, if the blue regions meeting one crossing circle agree, we may draw an embedded arc through that region with endpoints on the two intersections of the crossing circle. Attach to this the arc of intersection of the crossing disk $D$ with the projection plane.  This gives a simple closed curve $\gamma$ in the diagram of $K_{B,i}$ meeting the diagram twice. Push slightly off $D$ to ensure that the arc of $\gamma$ in the red surface meets no crossing disks.  By Lemma \ref{lemma:K2-0-prime}, $\gamma$ must have no crossings on one side.  But then we may isotope $D$ away from $K_{B,i}$.  This contradicts the fact that the diagram of $K$ is prime.

For \eqref{lmitm:assoc-cross}, draw arcs from either side of the crossing to the intersections of the crossing circle with the blue regions.  Connect these by the arc of intersection of the crossing disk with the projection plane.  Now twist, yielding $K$.  We have a closed curve of the diagram meeting a crossing coming from the twist region of the crossing circle, and also meeting a single crossing of the diagram.  Because $K$ is twist reduced, these crossings belong to the same twist region.  Hence the original crossing is associated with that crossing circle.
\end{proof}

The diagrams $K_{B,i}$, for $i=0, 2$ may not be twist reduced. However, we will need the fact that they are blue twist reduced, as defined below. 

\begin{define}\label{def:blue-twist-reduced}
A diagram of a link is \emph{blue twist reduced} if every simple closed curve in the projection plane meeting the diagram in exactly two crossings, with sides on the blue checkerboard surface, bounds a string of red bigons.
\end{define}

\begin{lemma}\label{lemma:blue-twist-reduced}
For $K$ a knot with prime, twist reduced, alternating diagram, the diagram of $K_{B,i}$, for $i=0, 2$, is blue twist reduced.
\end{lemma}

\begin{proof}
Suppose there exists a closed curve $\gamma$ in the diagram of $K_{B,i}$ meeting the diagram in blue regions, and meeting exactly two crossings.  By adding crossings to the diagram, giving red bigons, we obtain the diagram of $K$.  Since $\gamma$ is disjoint from the red surface, we may isotope it to be disjoint from crossing disks whose boundary crossing circle meets the blue surface.  Then when we add back in the crossings and red bigons to obtain the diagram of $K$, the curve $\gamma$ remains disjoint from the red surface, and meets the diagram of $K$ in exactly two vertices.  

Because $K$ is twist reduced, $\gamma$ bounds a string of bigons on one side in $K$, and they must be red bigons.  The crossings forming the bigons are all in the same twist region.  Either they remain in the diagram of $K_{B,i}$ when we remove crossings from $K$, or some of them are removed to obtain the diagram of $K_{B,i}$.  In the latter case, $i=2$ and $\gamma$ bounds a single bigon between the two crossings left in such a twist region in $K_{B,2}$.  In either case the lemma is proved.
\end{proof}

The following lemma will lead to contradictions in particular cases.

\begin{lemma}\label{lemma:22torus}
For $K$ a hyperbolic knot with a prime, twist reduced, alternating diagram, suppose $K_{B,i}$ is a $(2,2)$--torus link.  Then $L_{B,i}$ cannot have a crossing circle encircling the two crossings of $K_{B,i}$, for $i=0, 2$.
\end{lemma}

\begin{proof}
If $L_{B,i}$ has a crossing circle encircling the two crossings of a $(2,2)$--torus link, then both crossings are associated with that crossing circle.  When $i=0$, this contradicts the definition of $L_{B,0}$: it has at most one crossing associated with any crossing circle.

When $i=2$, the diagram of $L_{B,2}$ has at least one crossing associated with any crossing circle.  Since the diagram of $K_{B,2}$ has just two crossings, and both are associated with the given crossing circle, there can be no other crossing circles in $L_{B,2}$.  Hence when we put in the crossings to go from the diagram of $L_{B,2}$ to that of $K$, we only add crossings to the given twist region.  Thus $K$ is a $(2,2q)$--torus link.  This contradicts the assumption that $K$ is a hyperbolic alternating link.  
\end{proof}

\section{Surface properties}\label{sec:surfaces}
In this section, as well as the next two, we will give restrictions on the graphs $\Gamma_B$ coming from Lemmas~\ref{lemma:graph-bluevalence} and~\ref{lemma:graph-bdyincompr}, as well as a similar graph obtained in Lemma~\ref{lemma:htpc-graph}.  To analyze these graphs, we will actually be considering three surfaces in $S^3 \setminus L_{B,i}$.  The first surface is the blue surface $B_i$, which becomes $S_{B,i}$ in $S^3\setminus K$ by attaching annuli or M\"obius bands.  The second surface is the red surface $R_i$, which we have noted is a checkerboard surface for $K_{B,i}$, and is embedded in $S^3\setminus K$, $S^3\setminus L_{B,i}$, and $S^3\setminus K_{B,i}$.  The third surface we color \emph{green}.  It consists of all crossing disks bounded by the crossing circles of $L_{B,i}$. The green surface is embedded in $S^3\setminus L_{B,i}$, since each crossing disk is embedded and disjoint from the others.

\subsection{Graphs on a disk}\label{sec:graphs-on-a-disk}

Lemmas~\ref{lemma:graph-bluevalence} and~\ref{lemma:graph-bdyincompr} give us a graph $\Gamma_B$ on a disk $D$.
In fact, the results in this section, as well as in Sections~\ref{sec:trisquares} and~\ref{sec:boundary}, apply to any graph $\Gamma_B$ on a disk $D$ coming from the pull back of $S_{B,i}$ under a map $\phi\co D\to S^3\setminus {\rm{int}}(N(K))$, as in Lemmas~\ref{lemma:graph-bluevalence} and~\ref{lemma:graph-bdyincompr}, and in Lemma~\ref{lemma:htpc-graph} in Section~\ref{sec:homotopic}.

For all these graphs $\Gamma_B$, the subscript $B$ stands for \emph{blue} because  $\Gamma_B$ is the set of points in $D$ that map to the blue surface $S_{B,i}$.  Note that $\Gamma_B$ has a finite collection of isolated vertices on $D$ and $\bdy D$.  Hence removing these points, we obtain an embedding of a punctured disk $\phi'\co D' \to S^3\setminus L_{B,i}$.  The map $\phi'$ is transverse to the blue, red, and green surfaces in $L_{B,i}$.  Pulling back the intersections of these surfaces to $D'$, we obtain a graph $\Gamma_{BRG}$ with blue, red, and green edges on a punctured disk.  So, $\Gamma_{B}$ is a subgraph of $\Gamma_{BRG}$. There is also an intermediate subgraph $\Gamma_{BR}$ consisting of only the blue and red edges.

Each edge of these graphs maps to an arc in the diagram $L_{B,i}$. The arrangement of these arcs in this diagram will play a central role in this paper. Note that we are focusing on the diagram $L_{B,i}$ and hence every crossing circle punctures only blue regions.

We record now how corresponding edges may meet.  The proof follows immediately from the definitions, but we put the information into a lemma for future reference.

\begin{lemma}\label{lemma:edge-intersections}
Suppose $\phi\co D\to S^3\setminus K$ is a map of a disk meeting the blue, red, and green surfaces transversely, such that the pull back of these surfaces under $\phi$ gives an embedded planar graph $\Gamma_{BRG}$ with blue, red, and green edges, and the restriction $\phi'\co D'\to S^3\setminus L_{B,i}$ is a map of a punctured disk, with punctures mapping to crossing circles.
\begin{enumerate}
\item\label{lmitm:blue-blue} Blue edges meet blue edges only at vertices (punctures of $D'$).  In the diagram of $L_{B,i}$, these map to arcs that lie in blue regions of the complement of $K_{B,i}$ with endpoints on crossing circles.  Adjacent blue edges at a vertex map to arcs that meet the corresponding crossing circle in the two distinct regions of Lemma~\ref{lemma:blue-prime}\eqref{lmitm:distinct-regions}.  See Figure~\ref{fig:blue-edges}, (a) and (b). Note that these are vertices of $\Gamma_B$ and so the number of blue edges meeting at a vertex is provided by Lemma \ref{lemma:graph-bluevalence}. 
\item\label{lmitm:red-red-green-green} Because red and green surfaces are embedded, red edges do not meet red edges, and green edges do not meet green.
\item\label{lmitm:green-red} Green edges meet red on a crossing disk, along the arc in the plane of projection that runs between the two punctures of the twice--punctured disk.
\item\label{lmitm:green-blue} Green edges meet blue either in the interior of blue edges, which correspond to intersections in the interior of crossing disks, or at a vertex of $\Gamma_B$ (puncture of $D'$) where blue edges come together. Such a vertex corresponds to a crossing circle, and a green edge meeting this vertex corresponds to the green edge meeting the boundary of the crossing disk.
\item\label{lmitm:blue-red} Blue edges meet red edges at a crossing of the diagram.  If the blue and red edges together bound a region of $D'$ that is mapped to lie above the plane of projection, then the blue and red edges are mapped to meet the crossing as in Figure~\ref{fig:blue-red-crossing}(c).  If they bound a region mapped below the plane of projection, then they are mapped as in Figure~\ref{fig:blue-red-crossing}(d).  These are the only possibilities.
\item\label{lmitm:above-below}
The projection plane of the diagram of $K_{B,i}$ is made up of red and blue surfaces; we also refer to this as the projection plane of $L_{B,i}$.  Thus each region of $D'$ is mapped above or below the plane of projection of $L_{B,i}$, with regions switching from above to below or vice-versa across red or blue edges.
\end{enumerate}
\end{lemma}\qed

\begin{figure}
\begin{tabular}{lclclcl}
\input{blue-edges-graph-arxiv.tex} & \hspace{.2in} &
\input{blue-edges-diagram-arxiv.tex} & \hspace{.2in} &
  \includegraphics{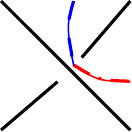} & \hspace{.2in} &
  \includegraphics{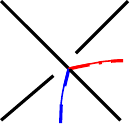} \\
  (a) && (b) && (c) && (d)
\end{tabular} 
\caption{Adjacent blue edges of $\Gamma_{BRG}$ meeting at a vertex as in (a) will lie in distinct regions as in (b).  In (c), blue and red edges meeting above the projection plane.  In (d), meeting below.}
\label{fig:blue-edges}
\label{fig:blue-red-crossing}
\end{figure}

Note that it is, in principle, possible for the inverse image of one of the red, green or blue surfaces to contain a simple closed curve component that is disjoint from the surfaces with other colours. However, any such component bounds a disk in one of the red, blue or green surfaces that is disjoint from the remaining surfaces, and so can easily be removed by a homotopy. We therefore assume that no such closed curves arise. Hence, as explained in the above lemma, the inverse image of the red, green and blue surfaces forms a graph $\Gamma_{BRG}$ in $D$.

We want to arrange this graph to be as simple as possible, in a certain suitable sense, so that various trivial arrangements can be ruled out.

\begin{define}\label{def:complexity}
Define the \emph{complexity} of $\phi'\co D'\to S^3\setminus L_{B,i}$ to be the ordered set
\[ C(\phi') = (\# \mbox{vertices}(\Gamma_B), \# \mbox{vertices}(\Gamma_{BR}), \# \mbox{vertices}(\Gamma_{BRG}), \# \mbox{edges}(\Gamma_{BRG}))
\]
\end{define}
Order complexity lexicographically.  We will assume that $\phi'$ has been chosen so that the complexity is as small as possible. 

\subsection{Monogons, bigons, and triangles}

In this subsection and the next, we give technical results with combinatorial proofs to show that certain configurations of the graph $\Gamma_{BRG}$ cannot hold.  Together, these results will give the Lemmas used in the proof of Theorems~\ref{thm:Bincompressible} and \ref{thm:bdryincompressible}.  

We will be considering complementary regions of the graph $\Gamma_{BRG}$ of red, blue, and green edges on $D'$.  The complementary regions in $D'$ need not be embedded in $S^3 \setminus L_{B,i}$, but they are disjoint from the red, blue, and green surfaces in $S^3\setminus L_{B,i}$.
We refer to those complementary regions that are disks by the number of edges they have, namely monogons meet one edge, bigons meet two, and triangles three. Sometimes, we will also consider subgraphs of $\Gamma_{BRG}$, for example, the subgraph $\Gamma_{BR}$ consisting only of the red and blue edges. We will also refer to complementary regions  of these subgraphs as monogons, bigons, triangles, and so on.

\begin{lemma}\label{lemma:red-green-bigons}
The graph $\Gamma_{BRG}$ has no bigons with one red side and one green, disjoint from blue.  More generally, there are no green edges disjoint from blue that have both endpoints on red.
\end{lemma}

\begin{proof}
The first statement follows immediately from the second, and so we prove the second.  The red surface runs through the interior of a crossing disk in $S^3\setminus L_{B,i}$.  Hence if a green edge on $D'$ has both endpoints on red, it has endpoints meeting the same embedded interval on the green 2--punctured disk.  Because it is disjoint from blue, the arc of intersection of green and red, along with the green edge, bounds a disk on the green surface.  Use this disk to homotope away the intersections with the red surface.  That is, use the disk to push the green arc of intersection to the other side of the projection plane.  In a neighborhood of the crossing disk, this will remove two red points of intersection with that crossing disk.  It will not affect any intersections of $D$ with crossing circles (vertices of $\Gamma_B$) or intersections of blue and red (vertices of $\Gamma_{BR}$). 
Hence, the graph $\Gamma_{BRG}$ in $D$ has been simplified: its number of vertices has been reduced.  This contradicts our minimality assumption on complexity.
\end{proof}

\begin{lemma}\label{lemma:blue-green-bigons}
The graph $\Gamma_{BRG}$ has no bigons with one green side and one blue, disjoint from red, and with at least one endpoint not being a vertex of $\Gamma_B$.  More generally, $\Gamma_{BRG}$ has no green edge disjoint from red with both endpoints on the same blue edge of $\Gamma_{BR}$ and  with at least one endpoint not being a vertex of $\Gamma_B$.
\end{lemma}

\begin{proof}
Such a green edge must lie on a single crossing disk.  The blue edge will either have both endpoints in the same region adjacent to the corresponding crossing circle, or it will have endpoints in the two distinct regions meeting that crossing circle.  In the latter case, a single blue edge that meets no red must belong to a single region, contradicting Lemma~\ref{lemma:blue-prime}\eqref{lmitm:distinct-regions}.  So the former must happen, that is, the blue edge has both of its endpoints in the same region adjacent to the crossing circle.
Then the green arc must bound a disk $E$ in the portion of the crossing disk on one side of the projection plane.  We may use this disk to homotope the disk $D$. A homotopy of $D$ along $E$ removes the intersection of the green and blue in the interior of the blue edge, without affecting the number of vertices of $\Gamma_B$ or $\Gamma_{BR}$, yet decreasing the number of vertices of $\Gamma_{BRG}$.  This contradicts the assumption that the complexity is minimal. 
\end{proof}

There is another option for a bigon with green and blue sides, namely when the green edge has its endpoints on two high valence vertices of $\Gamma_B$. In this case, the blue edge must be trivial (as in Definition \ref{def:trivial-edge} below) and we will deal with trivial blue edges and trivial bigons separately (as in Section \ref{sec:graphs}).

\begin{lemma}\label{lemma:red-blue-bigons}
The graph $\Gamma_{BR}$ has no bigons with one blue side and one red, whether or not the bigon meets the green surface.
\end{lemma}

\begin{proof}
Suppose there is a bigon with one red side and one blue.  Then it is mapped either completely above or completely below the projection plane.  Without loss of generality, say it is mapped above.  Then by Lemma~\ref{lemma:edge-intersections}\eqref{lmitm:blue-red}, the edges of the bigon meet the diagram of $K_{B,i}$ as shown in Figure~\ref{fig:red-blue-bigon}, left.  But then the union of these two edges forms a simple closed curve $\gamma$ meeting the diagram of $K_{B,i}$ twice with crossings on either side.  If $i=2$, or if $i=0$ and the red edge meets no green disks, this contradicts Lemma~\ref{lemma:K2-0-prime}.

\begin{figure}
  \includegraphics{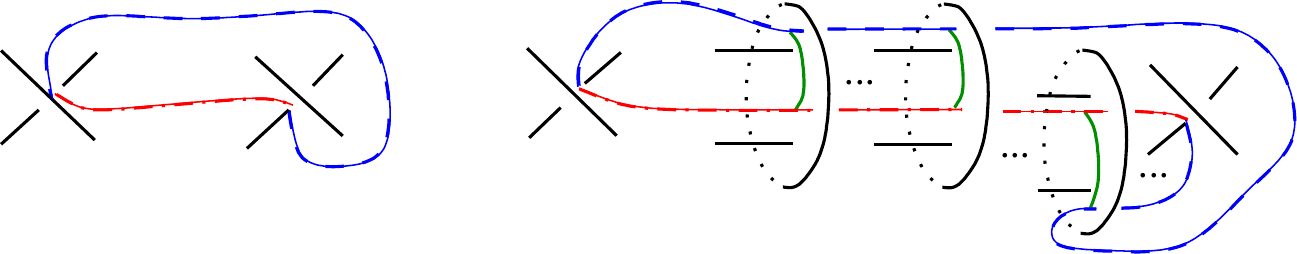}
  \caption{Left: A red--blue bigon that does not meet green. Right: case that it meets green}
  \label{fig:red-blue-bigon}
\end{figure}

So suppose $i=0$ and the red edge of $\gamma$ meets green disks.  Then Lemmas~\ref{lemma:red-green-bigons} and~\ref{lemma:blue-green-bigons} imply that green edges run from the blue edge to the red.   Thus the green disks meet the diagram in one of the two ways shown in Figure~\ref{fig:red-blue-bigon}, right, i.e.\ either with blue above or blue below.  

Follow $\gamma$ along the red edge.  Consider the first green disk $G$ that the red edge intersects.  Replace $\gamma$ with two new closed curves $\gamma_1$ and $\gamma_2$ by drawing an arc along $G$ from a red edge of $\gamma$ to a blue edge, and then splitting $\gamma$ along this arc.  Both closed curves $\gamma_1$ and $\gamma_2$ meet the diagram exactly twice.  One of them, say $\gamma_1$, bounds crossings on either side.  But note that $\gamma_1$ has one fewer points of intersection with the green disk than $\gamma$.  Thus, by induction on the number of intersections of the red edge with green, we obtain a contradiction to Lemma~\ref{lemma:K2-0-prime}. 
\end{proof}

\begin{lemma}\label{lemma:scc}
Consider the blue graph $\Gamma_B$ alone.  There are no simple closed curves of intersection of the blue surface.  That is, $\phi(D)$ does not meet the blue surface in any component disjoint from crossing circles.
\end{lemma}

\begin{proof}
Suppose there is a simple closed curve of intersection of the blue surface.  By Lemmas~\ref{lemma:blue-green-bigons} and \ref{lemma:red-blue-bigons}, it cannot meet either the green or red surface, for an outermost arc of intersection would give an illegal bigon. So, it is a blue simple closed curve disjoint from the red and green surfaces.  But as explained above, we have arranged that $D$ contains no such curves.
\end{proof}

\begin{lemma}\label{lemma:no-monogons}
In the blue graph $\Gamma_B$, there are no monogons.  That is, no edges of intersection run from one vertex (corresponding to a crossing circle) back to that same vertex forming a monogon.
\end{lemma}

\begin{proof}
Lemmas~\ref{lemma:blue-green-bigons} and \ref{lemma:red-blue-bigons} imply that any blue edge forming a monogon in $\Gamma_B$ cannot meet red or green edges of $\Gamma_{BRG}$ in its interior.  Hence the edge of a monogon runs from one side of a crossing circle through a region of the projection plane to the other side.  This implies that there is a single complementary region of the diagram on both sides of a crossing circle, contradicting Lemma~\ref{lemma:blue-prime}\eqref{lmitm:distinct-regions}.
\end{proof}

The above two lemmas imply there are no simple closed curves or monogons in $\Gamma_B$.  There will be bigons, and these can be trivial or non-trivial, as mentioned in Section~\ref{sec:graphs}.  We are now ready to define trivial edges and trivial bigons.  Since Lemma~\ref{lemma:adj-bigons} implies not all bigons can be trivial, we will then move to studying non-trivial bigons. 

\begin{define}\label{def:trivial-edge}
Define an edge of $\Gamma_{BRG}$ to be \emph{trivial} if it is a blue arc that is disjoint from the red edges, and its endpoints are distinct vertices of $\Gamma_B$ in $D$, but correspond to the same crossing circle in $L_{B,i}$.
\end{define}

Let $\alpha$ be a trivial edge. Since it is a blue edge disjoint from the red edges of $\Gamma_{BRG}$, it corresponds to an arc in a blue region of the diagram $L_{B,i}$. By assumption, this arc has both its endpoints on the same crossing circle, and the two blue regions at the punctures of this crossing circle are distinct. Hence, $\alpha$ must have both its endpoints on the same puncture. It therefore forms a closed loop in $L_{B,i}$, bounding a disk $E$ in the diagram with interior which is disjoint from the red, blue and green surfaces and which can only meet green in single edges with both endpoints on high valence blue vertices. See Figure~\ref{fig:trivialbigons3}.

\begin{figure}
 \includegraphics[width=2in]{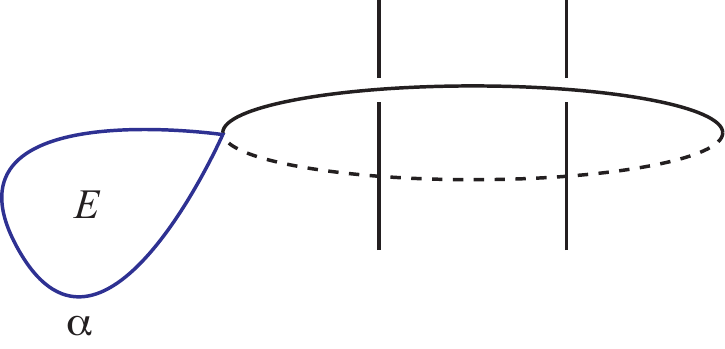}
\caption{A blue arc starting and ending on the same crossing circle}
\label{fig:trivialbigons3}
\end{figure}

There are four types of trivial arcs: those with neither endpoint on $\partial D$, those with exactly one endpoint on $\partial D$, those with both endpoints on $\partial D$, and those lying entirely in $\partial D$. We will deal with these types in different ways.

\begin{lemma}\label{no-interior-trivial-arcs}
In the graph $\Gamma_{BRG}$, there are no trivial arcs with both endpoints in the interior of $D$.
\end{lemma}

\begin{proof}
As explained above, such an arc $\alpha$ bounds a disk $E$ in the diagram $L_{B,i}$ with interior that is disjoint from the red, blue and green surfaces.  We may perform a homotopy to $D$, sliding $\alpha$ along $E$,  so that it ends up in $\partial N(C)$, where $C$ is the crossing circle at the endpoints of $\alpha$.  (See Figure \ref{fig:crossingcirclehomotopy}.)  The two vertices at the endpoints of $\alpha$ lie in meridian disks of $N(C)$ that are subsets of $D$.  The union of these two meridian disks with a regular neighbourhood of $\alpha$ is a disk $D'$ in the interior of $D$, that maps to $N(C)$.  Now $\partial D'$ is a curve on the boundary of the solid torus $N(C)$ that is homotopically trivial.  Hence, we may homotope $\partial D'$ so that it is a multiple of a meridian curve for $N(C)$.  If this multiple is zero, then we may homotope $D'$ to lie in $\partial N(C)$.  It thereby misses $C$, and so this reduces the number of vertices of $\Gamma_B$.  This contradicts our assumption that complexity is minimal.
On the other hand, if $\partial D'$ represents a non-zero multiple of a meridian, then we may homotope $D'$ so that it maps to the crossing circle at a single point, where it forms a branch point. In this way, we end up with a single vertex of $\Gamma_B$, and again we have reduced its complexity. Again we reach a contradiction.
\end{proof}

Note that the procedure described in the above proof may create vertices with valence that is a non-zero multiple of $n_j$. It was for this reason that vertices of this form are permitted in Lemmas~\ref{lemma:graph-bluevalence} and~\ref{lemma:graph-bdyincompr}.

\begin{figure}
 \includegraphics{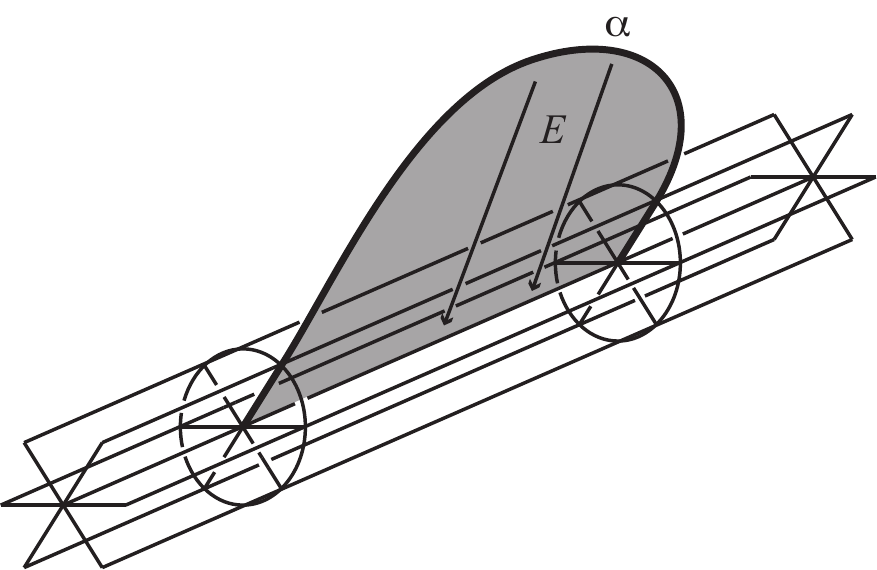}
\caption{Modifying $D$ by a homotopy}
\label{fig:crossingcirclehomotopy}
\end{figure}

\begin{lemma}\label{lemma:no-boundary-trivial-arcs}
In the graph $\Gamma_{BRG}$, there are no trivial arcs that are subsets of $\partial D$.
\end{lemma}

\begin{proof}
We may perform a similar homotopy to the one in Lemma~\ref{no-interior-trivial-arcs}, but we may take the disk $E$ to be a subset of $S_{B,i}$. In this way, $\partial D$ remains in $S_{B,i}$, but the complexity of $D$ is reduced. Again, this is a contradiction.
\end{proof}

Finally, suppose that there is a trivial arc in $\Gamma_{BRG}$ properly embedded in $D$ with one or two endpoints on the boundary of $D$. There is no obvious way of eliminating such an arc, and so we deal with them another way in the proof of Lemma~\ref{lemma:adj-bigons}.

\begin{lemma}\label{all-but-one-edge-trivial}
If all but one of the edges of a region in $\Gamma_B$ are trivial, then the remaining edge is also trivial.
\end{lemma}

\begin{proof}
Consider such a region in $\Gamma_B$. The trivial edges in the boundary of the region are disjoint from the red and green surfaces. So, if any red arc enters the region, it must intersect the remaining blue side twice. Hence, an outermost such arc gives a bigon in $\Gamma_{BR}$ with one red side and one blue side, contradicting Lemma \ref{lemma:red-blue-bigons}. So, there are no red arcs in the region. If there are any green arcs intersecting the remaining blue edge in its interior, then we similarly deduce that there is a bigon in $\Gamma_{BRG}$ with one blue side and one green side, and which is disjoint from the red edges, contradicting Lemma \ref{lemma:blue-green-bigons}. We therefore deduce that all the edges  of the region are disjoint from the red and green surfaces. The trivial edges start and end at the same crossing circle. Hence, the remaining edge does also. It is therefore trivial. 
\end{proof}

\begin{lemma}\label{lemma:no-blue-bigons}
In the graph $\Gamma_{BRG}$, there are no non-trivial bigons with two blue sides, disjoint from red and green edges.
\end{lemma}

\begin{proof}
For such a bigon, the vertices would correspond to crossing circles of $L_{B,i}$.  Since the bigon is non-trivial, the vertices correspond to distinct crossing circles.  Put in the crossings at the two crossing circles, and elsewhere, to form the diagram of $K$.  The blue edges are disjoint from crossing disks, and so they remain on the projection plane in the complement of the diagram of $K$.  They connect across (former) crossing disks to give a simple closed curve in the diagram of $K$ that meets the diagram at exactly two crossings.  Because $K$ is twist reduced, these crossings correspond to the same twist region.  But then they would have corresponded to the same crossing circle in $L_{B,i}$.  This is a contradiction.
\end{proof}


\begin{lemma}\label{lemma:RGB-adjacent-green}
In $\Gamma_{BRG}$, there are no pairs of red--green--blue triangles (RGB triangles) adjacent across a green edge.
\end{lemma}

\begin{proof}
Such a pair of triangles would give a red--blue bigon in $\Gamma_{BR}$, contradicting Lemma~\ref{lemma:red-blue-bigons}.
\end{proof}

\begin{lemma}\label{lemma:RGB-adjacent-blue}
In $\Gamma_{BRG}$, there are no pairs of RGB triangles adjacent across a blue edge.  More generally, no pair of green edges can be added to the graph $\Gamma_{BR}$ in such a way that the result is a pair of triangles adjacent across a blue edge.  
\end{lemma}

\begin{proof}
If a pair of green edges exists as in the statement of the second part of the lemma, then there exists an innermost such pair.  That is, red, blue, and green edges of $\Gamma_{BRG}$ bound two triangular regions between them with red, green, and blue sides, identified along the blue side.  In addition, there may be other green edges in $\Gamma_{BRG}$ that intersect these triangular regions, but the fact that it is innermost implies none of the green edges cut off a pair of such triangular regions adjacent across the blue side.

Because there are no additional red or blue edges inside these triangular regions, one of the triangular regions is mapped above the plane of projection, one below.  Because there is only one vertex where red meets blue in the two triangles, there would be just one crossing in the diagram where the triangles meet, and the edges must run from the crossing as shown in Figure \ref{fig:RGB-adjacent-blue}, middle.  The green arc lies on a single crossing circle, although one edge lies above the plane of projection and one below.  Hence the two red edges running from the crossing in the figure must run to the same crossing disk, and in fact to the same arc on the projection plane between the two punctures of the crossing disk.

\begin{figure}
  \input{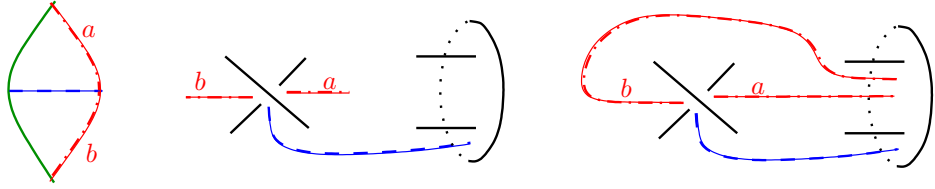}
  \caption{Two RGB triangles adjacent across a blue edge lead to the
    configurations shown.}
  \label{fig:RGB-adjacent-blue}
\end{figure}

If the red edges have endpoints on opposite sides of the crossing disk, then one of the triangles links the crossing circle nontrivially, contradicting the fact that it bounds a disk disjoint from the crossing disk.  Hence the red edges have endpoints on the same side of the crossing disk.  But then we can form a simple closed curve meeting the diagram of $K_{B,i}$ twice with crossings on either side, by taking the union of the red edge labeled $b$, the blue edge, and a portion of the crossing disk in Figure~\ref{fig:RGB-adjacent-blue}, right.  

By Lemma~\ref{lemma:K2-0-prime}, we must have $i=0$ and this red edge must meet another crossing disk, i.e.\ another component of the green surface.  This gives a new green edge in the graph, which runs from a red edge in an RGB triangle.  It can't run to the green side of the triangle, since green edges don't meet green.  Thus it runs to blue.  But then the green arc continues into the other RGB triangle, and must exit through red, and we have another pair of RGB triangles adjacent across the blue, contained in the original pair.  This contradicts the fact that the pair was innermost.
\end{proof}

\begin{lemma}\label{lemma:NoGreenLoop} 
No green edge of $\Gamma_{BRG}$ can have both its endpoints on the same vertex of $\Gamma_B$.  
\end{lemma}

\begin{proof}  Such a green edge bounds a disk $E$ on a crossing disk.  Use this disk to homotope the disk $D$ past the crossing disk, removing a green edge of $\Gamma_{BRG}$, contradicting the fact that the graph had minimal complexity.
\end{proof}

\begin{lemma}\label{lemma:noBBG}
The graph $\Gamma_{BRG}$ has no blue--blue--green triangles.  More generally, no green edge can be added to $\Gamma_{BR}$ to cut off a triangular region with two blue sides and one green.  
\end{lemma}

\begin{proof}
Suppose some green edge can be added to $\Gamma_{BR}$ to cut off such a triangular region.  Then there must be an innermost such green edge, so that by Lemmas \ref{lemma:NoGreenLoop} and \ref{lemma:blue-green-bigons}, the triangular region it bounds is disjoint from all other green edges.  Hence it suffices to prove the first claim.  

The two blue edges of such a triangle meet at a vertex corresponding to a crossing circle, bounding a crossing disk.  The green edge lies on a crossing disk.  If these crossing disks agree, then each of the blue edges runs from a crossing circle back to the same crossing disk.  We may form two closed curves in the projection plane, one on either side of the crossing disk, by connecting the endpoints of the blue edge together in the projection plane, and we may do so in such a way that the closed curves bound disks in the projection plane disjoint from the diagram of $L_{B,i}$.  We can use one of these two disks to homotope the triangle through the crossing circle.  In the graph $\Gamma_{BRG}$, this has the effect of sliding the endpoint of the green edge so that it ends at the blue vertex.
This reduces the number of vertices of the graph $\Gamma_{BRG}$ without affecting the number of vertices of $\Gamma_B$ and $\Gamma_{BR}$, which contradicts our minimality assumption.

So suppose the crossing disk of the green edge has boundary on a different crossing circle than that corresponding to the vertex.  Then as we argued in Lemma~\ref{lemma:no-blue-bigons}, the blue edges connected to the crossing circles can be adjusted to give a closed curve in the diagram of $K$ meeting the diagram at exactly two crossings, each in a distinct twist region of $K$.  This gives a contradiction to the fact that the diagram of $K$ is twist reduced.
\end{proof}

\subsection{Adjacent blue--blue--red triangles}

In the following lemmas, we will consider adjacent blue--blue bigons meeting red edges, which will lead to considering blue--blue--red triangles by Lemma~\ref{lemma:red-blue-bigons}.  In this subsection, we will give restrictions on adjacent blue--blue--red triangles.

\begin{lemma}\label{lemma:no-green-meets-bcj}
Suppose the graph $\Gamma_{BR}$ contains three adjacent blue--blue--red triangles, adjacent at the vertex meeting blue edges, with edges labelled as in Figure \ref{fig:BBR-labels}.  Then when we include again the green edges, no green edge meets edge $b$ or $c$.  Moreover, if a green edge meets the edge $j$, then it must run from the common vertex of the triangles to $j$.  In particular, it comes from an intersection of the disk $D$ with the crossing disk associated to the vertex.
\end{lemma}

\begin{figure}
  \input{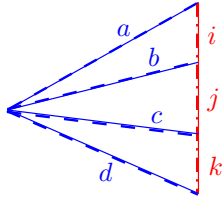}
  \caption{Labels on three adjacent blue--blue--red triangles, adjacent at the blue vertex.}
  \label{fig:BBR-labels}
\end{figure}

\begin{proof}
If green meets $b$, then a green edge runs through the triangles $abi$ and $bcj$ in Figure~\ref{fig:BBR-labels}.  By Lemma~\ref{lemma:noBBG} (no blue--blue--green triangles), a green edge meeting $b$ cannot meet either $a$ or $c$.  By Lemma~\ref{lemma:blue-green-bigons} (no blue--green bigons), it cannot run from $b$ back to $b$, or even to the vertex meeting the blue edges.  Hence it must run to $i$ and to $j$.  But this contradicts Lemma~\ref{lemma:RGB-adjacent-blue} (no RGB triangles adjacent along blue).  So green cannot meet $b$.  The argument for $c$ is symmetric.

If green meets $j$, then a green edge runs through the triangle with edges $b$, $c$, and $j$.  Since green cannot meet $b$ or $c$, and cannot form a green--red bigon, it must run directly to the vertex.
\end{proof}

Note that in a single blue--blue--red triangle, there are two points where blue edges meet red, each of which corresponds to a crossing of the diagram, by Lemma~\ref{lemma:edge-intersections}\eqref{lmitm:blue-red}.

\begin{lemma}\label{lemma:distinct-crossings-2}
Restrict to the graph $\Gamma_{BR}$. In the case $i=2$, a single blue--blue--red triangle meets two distinct crossings.
\end{lemma}

\begin{proof}
If not, then the red edge of the triangle runs from one crossing back to the same crossing.  By Lemma~\ref{lemma:edge-intersections}\eqref{lmitm:blue-blue}, the blue edges run to opposite sides of that crossing.  Then since the triangle is either mapped entirely above or entirely below the plane of projection, Lemma~\ref{lemma:edge-intersections}\eqref{lmitm:blue-red} implies that the endpoints of the red edge are on opposite sides of the crossing as well.  But this contradicts Lemma~\ref{lemma:blue-prime}\eqref{lmitm:distinct-red}, for the diagram $L_{B,2}$.
\end{proof}

In the case $i=0$, the two crossings at endpoints of the red edge in a red--blue--blue triangle may actually not be distinct.  However, when we have three adjacent triangles, we are able to rule out too much overlap of crossings.

\begin{lemma}\label{lemma:distinct-crossings}
Restrict to the graph $\Gamma_{BR}$.  If three blue--blue--red triangles are adjacent, then each triangle must meet two distinct crossings.
\end{lemma}

\begin{proof}
In the case $i=2$, this is immediate from Lemma~\ref{lemma:distinct-crossings-2}.  So we will only consider the case $i=0$, and the diagram $L_{B,0}$.
  
Label the triangles as in Figure \ref{fig:BBR-labels} again.  Without loss of generality, we may assume that triangles $abi$ and $cdk$ map above the plane of projection, and triangle $bcj$ maps below.

First, we prove that triangle $bcj$ cannot meet exactly one crossing.  Suppose by way of contradiction that it does.  Then edges $b$ and $c$ of the triangle must meet opposite sides of this crossing, and endpoints of $j$ must also meet opposite sides of the crossing as shown in Figure~\ref{fig:distinct-crossings-1}, left.

\begin{figure}
  \input{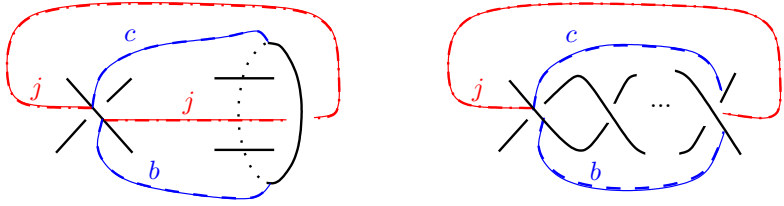}
  \caption{If triangle $bcj$ meets just one crossing, edges $b$, $c$, and $j$ meet the crossing as shown left. Right: this gives a simple closed curve in the diagram of $K$ meeting the diagram twice.}
  \label{fig:distinct-crossings-1}
\end{figure}

It follows that $j$ runs through the crossing disk $G$ corresponding to the crossing circle of the vertex, as in Figure~\ref{fig:distinct-crossings-1}.  By Lemma~\ref{lemma:no-green-meets-bcj}, no other green edge can meet $j$.  But now put the crossings of $K$ back into the diagram, to obtain the diagram as in Figure~\ref{fig:distinct-crossings-1}, right.  Note that the union of $j$ and $b$ form a closed curve meeting the diagram of $K$ twice.  This contradicts the fact that the diagram of $K$ is prime and hyperbolic.

So if one of the triangles meets exactly one crossing, it cannot be the middle one.  It must be triangle $abi$ or triangle $cdk$.  The two cases are symmetric, and so we show that triangle $abi$ meets two crossings.  If not, then as in the previous argument, edges $a$ and $b$ meet opposite sides of the same crossing, call it $x$, and edge $i$ must run through the crossing disk $G$ to meet $x$ on either side.  Again if $i$ meets no other green crossing disks, then by putting in twists corresponding to the twist region shown, we would obtain a contradiction to the fact that the diagram of $K$ is prime and hyperbolic, just as in the above paragraph.  Thus we must conclude that $i$ meets some other green disk, a possibility which was ruled out for $j$ by Lemma~\ref{lemma:no-green-meets-bcj}, but which is not impossible for $i$.

Thus there must be some green edge $g'$, corresponding to a green disk $G'$ distinct from the crossing disk $G$ at the vertex of $a$, $b$, and $c$, such that $g'$ lies in the triangle $abi$ and meets $i$.  Note that $g'$ cannot have its other endpoint at the vertex where $a$ and $b$ meet, since $G'$ is distinct from $G$. By Lemmas~\ref{lemma:no-green-meets-bcj} and \ref{lemma:red-green-bigons}, the other endpoint of $g'$ meets $a$, forming a triangular region of $\Gamma_{BRG}$.

Consider such a triangle nearest the vertex where $a$ and $i$ meet.  That is, take $g'$ such that the triangle cut off by $g'$ and portions of $a$ and $i$ in $\Gamma_{BRG}$ contains no other green edges disjoint from $G$. Notice that with this choice of $g'$, the triangular region either contains no other green edges, or it contains a green edge $g$ coming from the crossing disk $G$, cutting off an interior triangle.

If $g'$ cuts off a triangle in $\Gamma_{BRG}$ that contains no other green edges, then the boundary of that triangle maps to give a curve in the projection plane meeting the diagram of $K_{B,0}$ twice, disjoint from all crossing disks.  If instead $g'$ cuts off a triangle meeting a green edge $g$, then the boundary of the quadrilateral with sides $g'$, $g$, and portions of $a$ and $i$, maps to give a closed curve in $K_{B,0}$ meeting the diagram exactly twice.  In either case, by Lemma~\ref{lemma:K2-0-prime}, the curve bounds no crossings.  This is shown on the left of Figure~\ref{fig:distinct-crossings-2} for the triangle case.  Draw an arc in the diagram running over the strand of the knot between $a$ and $i$, and label this arc $\ell$.

\begin{figure}
  \input{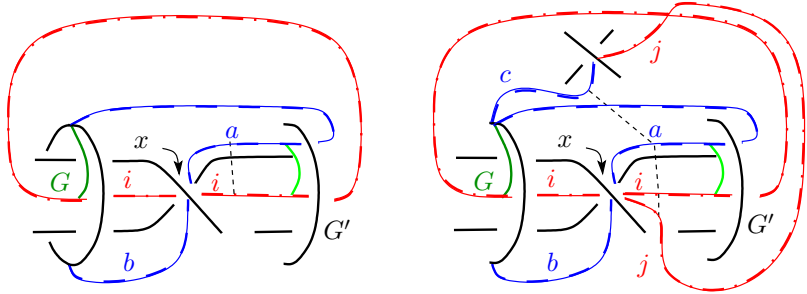}
  \caption{Possible configuration of diagram when triangle $abi$ meets just one crossing. The thin dashed line is the arc $\ell$ in the proof.}
  \label{fig:distinct-crossings-2}
\end{figure}

Now put triangle $bcj$ into the figure.  One endpoint of $j$ meets the crossing $x$ on the side of $b$ determined by Lemma~\ref{lemma:edge-intersections}\eqref{lmitm:blue-red}.  By Lemma~\ref{lemma:no-green-meets-bcj}, $j$ does not meet the crossing disk $G'$.  This is shown on the right of Figure~\ref{fig:distinct-crossings-2}.  Because $i$ and $j$ are in the same region, we may connect an endpoint of the arc $\ell$ to the edge $j$ without meeting the diagram of $K_{B,0}$.  Similarly, because $a$ and $c$ are in the same region, we may connect the other endpoint of the arc $\ell$ to the edge $c$ without meeting the diagram.  The arc $\ell$ is shown by the thin dashed line in Figure~\ref{fig:distinct-crossings-2}.

Then the union of the arc $\ell$, the portion of $j$ running from $\ell$ to the crossing where it meets $c$, and the portion of $c$ running back to $\ell$, forms a closed curve $\gamma$ in the diagram meeting the diagram of $K_{B,0}$ exactly twice.  Observe that $\gamma$ is disjoint from the crossing disk $G$.  Lemma~\ref{lemma:K2-0-prime} implies that the portion of $\gamma$ in the red face must run through a crossing disk, which implies $j$ must run through a crossing disk, distinct from $G$.  But Lemma~\ref{lemma:no-green-meets-bcj} implies that $j$ can only run through $G$.  Together, these give a contradiction.
\end{proof}

\begin{lemma}\label{lemma:adjacent-BBR-triples}
The graph $\Gamma_{BR}$ cannot contain two pairs of three adjacent blue--blue--red triangles adjacent across the red edges, as in Figure~\ref{fig:adjacent-BBR-triples}.
\end{lemma}

\begin{figure}
  \input{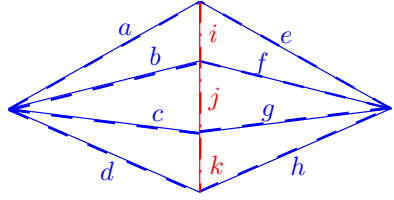}
  \caption{Labeling for adjacent blue--blue--red triples}
  \label{fig:adjacent-BBR-triples}
\end{figure}

\begin{proof}
Suppose not.

We first argue that the two crossing circles corresponding to the two vertices where blue edges meet cannot agree.  For if they do, then $a$, $c$, and $f$ are in the same region, and $b$, $e$, and $g$ are in the same region, and so the crossings at endpoints of $a$ and $e$, at endpoints of $b$ and $f$, and at endpoints of $c$ and $g$ are all associated with the crossing circle (Lemma~\ref{lemma:blue-prime}\eqref{lmitm:assoc-cross}). By Lemma~\ref{lemma:distinct-crossings}, at least two of those crossings are distinct.   This immediately gives a contradiction for the diagram of $L_{B,0}$, by Definition~\ref{def:L0-L2}.   However, if two of those crossings agree, then we do not immediately have a contradiction in the case of $L_{B,2}$.  However, in $L_{B,2}$, the following unions of edges will give simple closed curves in the diagram of $K_{B,2}$ meeting the diagram twice: $c\cup i\cup f$, $e \cup i \cup b$, $f\cup a\cup j$, $b\cup j\cup g$.  Because the diagram of $K_{B,2}$ is prime (Lemma~\ref{lemma:K2-0-prime}), these curves encircle portions of the diagram meeting no crossings, and the diagram of $K_{B,2}$ is that of a $(2,2)$--torus link, with a crossing circle of $L_{B,2}$ encircling the two crossings.  This contradicts Lemma~\ref{lemma:22torus}.  Thus the two vertices where blue edges meet correspond to two distinct crossing circles.

It follows that neither of the crossing disks corresponding to the vertices can intersect the edge $j$, for by Lemma~\ref{lemma:no-green-meets-bcj}, if the green meets $j$, then a green edge must run from $j$ to the vertex where $b$ and $c$ meet, and a green edge must run from $j$ to the vertex where $g$ and $f$ meet.  This would imply that those two vertices correspond to the same crossing disk, which we showed cannot happen.

Additionally, we claim that we cannot have both crossing disks meeting the edge $i$.  For if the crossing disk corresponding to the vertex of $e$, $f$, and $g$ meets $i$, then there is a corresponding green edge in the triangle $abi$ with an endpoint on $i$.  By Lemma~\ref{lemma:no-green-meets-bcj}, its other edge meets $a$.  Then $a$ is in the same region as $f$.  Similarly, if the crossing disk corresponding to the vertex of $a$, $b$, and $c$ meets $i$, then $e$ is in the same region as $b$.  Lemma~\ref{lemma:blue-prime}\eqref{lmitm:assoc-cross} implies that the crossing at the endpoints of $a$ and $e$, as well as the crossing at the endpoints of $b$ and $f$, must both be associated to both crossing circles.  But any single crossing is associated to at most one crossing circle.  This contradiction implies that we cannot have both crossing disks meeting $i$.
Without loss of generality, assume that the crossing disk corresponding to the vertex at $a$, $b$, and $c$ does not meet $i$.

Now we claim that the crossing $x$ at the endpoints of $a$ and $e$ and the crossing $y$ at the endpoints of $c$ and $g$ must be distinct.  For if not, then $i \cup j$ maps to a closed curve in the diagram, and the two crossing circles must link $i \cup j$.  But then the crossing disk corresponding to the vertex where $a$, $b$, and $c$ meet would have to intersect either $i$ or $j$, and this contradicts the above arguments.

Thus we have established:  (1)  The crossings $x$ and $y$ are distinct.  (2)  The crossing disk corresponding to the vertex meeting $a$, $b$, and $c$ does not intersect $i$ or $j$.  By Lemma~\ref{lemma:distinct-crossings}, the crossing at the endpoints of $b$ and $f$ is also distinct from $x$ and $y$.  Hence we may sketch images of the triangles $abi$ and $bcj$, and endpoints of edges $e$, $f$, and $g$ as in Figure~\ref{fig:BBR-triples-1}, left.

\begin{figure}
  \input{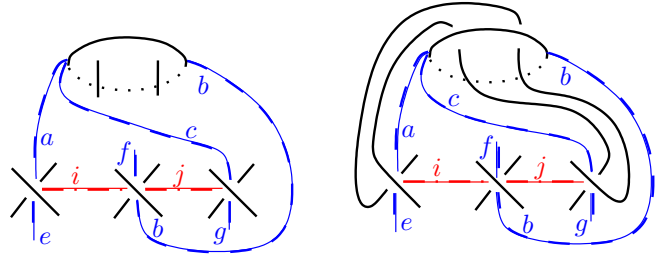}
  \caption{Left: configuration of three adjacent blue--blue--red
    triangles when the triangles meet three distinct crossings. Right:
    diagram of $L_{B,2}$ in this case.}
  \label{fig:BBR-triples-1}
\end{figure}

Note that $g$ starts within the region of the diagram bounded by $b\cup j\cup c\cup (\mbox{crossing disk})$ and $e$ starts outside it.  Now $e$ and $g$ share a vertex that is mapped to the same side of the same crossing circle. Hence, $e$ and $g$ lie in the same blue region of the diagram.  Neither $e$ nor $g$ can meet the region of $c$, by Lemma~\ref{lemma:blue-prime}\eqref{lmitm:distinct-regions-cross}. They cannot meet $j$ because they are the wrong color.  Hence one must meet $b$ or the crossing disk in the region of $b$, and so both are in the same region as $b$.  Then Lemma~\ref{lemma:blue-prime}\eqref{lmitm:assoc-cross} implies that the crossing $x$ at the endpoints of $a$ and $e$ and the crossing $y$ at the endpoints of $c$ and $g$ both are associated with the crossing circle corresponding to the vertex of $a$, $b$, and $c$.  This contradicts Definition~\ref{def:L0-L2} in the case of $L_{B,0}$: each crossing circle is associated with at most one crossing.

So we continue for the diagram $L_{B,2}$ only.  In Figure~\ref{fig:BBR-triples-1} right, the diagram of $L_{B,2}$ is shown, with crossings associated with the crossing circle attached to that crossing circle, and with the crossing circle cutting through the bigon formed by the two crossings, as required in Definition~\ref{def:L0-L2}.  Now note that the edge $f$ is enclosed by a region bounded by strands of the diagram, as well as red edges $i$ and $j$.  The edge $f$ cannot cross any of these.  Similarly, $e$ and $g$ cannot cross these.  Thus the crossing circle corresponding to the vertex of $e$, $f$, and $g$, which is distinct from the crossing circle corresponding to the vertex of $a$, $b$, and $c$, must straddle either $i$ or $j$.  It cannot straddle $j$, by the above work, so it straddles $i$.  Then a green edge with endpoint on $i$ meets triangle $abi$.  By Lemma~\ref{lemma:no-green-meets-bcj}, its other endpoint is on $a$.  Then $f$ is in the region of $a$ and $c$.  By Lemma~\ref{lemma:blue-prime}\eqref{lmitm:assoc-cross}, all three crossings shown are associated with the crossing circle of $a$, $b$, and $c$.  This contradicts the definition of $L_{B,2}$, Definition~\ref{def:L0-L2}.
\end{proof}

\section{Triangles and squares}\label{sec:trisquares}

The main result of this section is that the graph $\Gamma_B$ of Lemma~\ref{lemma:graph-bluevalence} cannot contain three and five adjacent bigons, for $L_{B,0}$ and $L_{B,2}$, respectively.  To prove this, we examine adjacent triangles and squares.

Suppose a blue--blue--red triangle in $\Gamma_{BR}$ is adjacent to a blue and red square, adjacent across the red edge, as illustrated on the left of Figure~\ref{fig:tri-squares-labels}.  We will call this a \emph{triangle--square pair}.  In this section, we examine adjacent triangle--square pairs, adjacent at the vertex of the triangle.  We show that in $L_{B,0}$ there cannot be three adjacent triangle--square pairs, and in $L_{B,2}$ there cannot be five adjacent triangle--square pairs.

\begin{figure}
  \includegraphics{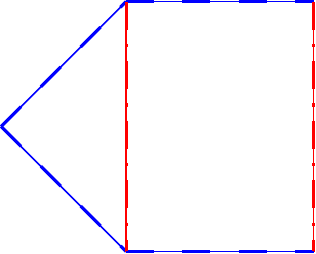} \hspace{.5in}
 \input{trisquare-labels-arxiv.tex}
  \caption{Left: A triangle--square pair.  Right: Three adjacent
    triangle--square pairs will be labeled as shown.}
  \label{fig:tri-squares-labels}
\end{figure}

Label three adjacent triangles and squares as in Figure~\ref{fig:tri-squares-labels}.  By Lemma~\ref{lemma:distinct-crossings}, each triangle shown meets distinct crossings.  Thus the diagram may have two, three, or four crossings coming from the intersections of blue and red edges of triangles.  Up to symmetry, there are exactly four possibilities, shown in Figure~\ref{fig:trisquare-options}.

\begin{figure}
  \input{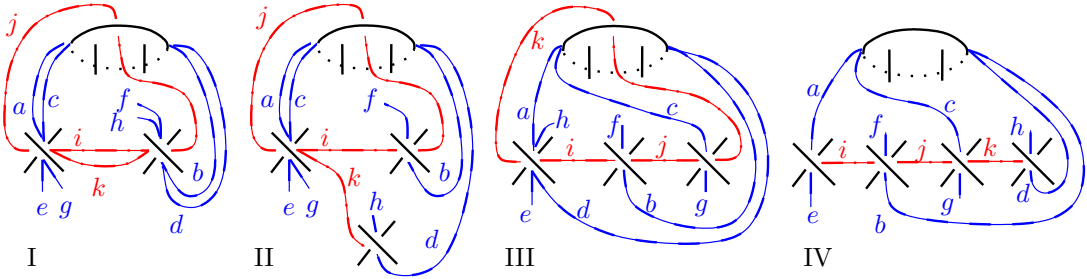}
  \caption{Possibilities for three adjacent triangles.}
  \label{fig:trisquare-options}
\end{figure}

Note the diagrams shown are only sketches.  In options I and II, the curve $i \cup j$ must run through the crossing disk, and in option III, the curve $i\cup j\cup k$ must run through the crossing disk.  We have chosen to show $j$ running through the disk in I and II, and $k$ in III.  Additionally, we have chosen to show none of $i$, $j$, or $k$ running through the crossing disk in option IV.  However, a priori, these are arbitrary choices.  Hence we must be careful to ensure that our arguments below do not depend on these choices.

For all four possibilities, we have the following lemma.

\begin{lemma}\label{lemma:fg-endpoints}
Suppose the edge labeled $f$ in Figure \ref{fig:tri-squares-labels} is not mapped into the region of the diagram of $K_{B,i}$ containing images of edges labeled $a$ and $c$.  Then $f$ cannot have both of its endpoints mapped to the same crossing of $K_{B,i}$.  Similarly, if $g$ is not mapped into the same region as edges labeled $b$ and $d$, then $g$ cannot have both of its endpoints mapped to the same crossing.
\end{lemma}

\begin{proof}
We will prove the lemma for $f$.  The argument for $g$ is symmetric.  Again we will refer to an edge in Figure~\ref{fig:tri-squares-labels} and its image in the diagram of $L_{B,i}$ by the same name.

Suppose $f$ is not in the region of $a$ and $c$, but $f$ has both endpoints on the same crossing.  Then $f$ forms a loop in a blue region.
By our assumption that complexity is minimal, $f$ cannot be homotoped away from this blue region.  Thus it meets a crossing disk $G$.  Note that since $f$ does not meet the region of $a$ and $c$ by assumption, and $f$ cannot meet the region of $b$ by Lemma~\ref{lemma:blue-prime}\eqref{lmitm:distinct-regions-cross}, $G$ is disjoint from the crossing disk corresponding to the vertex at $a$, $b$, $c$ and $d$.

First, we claim that $G$ meets the edge $i$.  Because $f$ intersects $G$, there are green edges $\gamma_1$ and $\gamma_2$, each with one endpoint on $f$, running through squares $eif\ell$ and $fjgm$ of Figure~\ref{fig:tri-squares-labels}, respectively, corresponding to intersections with $G$.  If $\gamma_1$ meets $i$, then the claim is proved.  So suppose $\gamma_1$ meets $e$ or $\ell$.  By Lemma~\ref{lemma:no-green-meets-bcj}, $\gamma_2$ cannot meet $j$, so $\gamma_2$ meets either $m$ or $g$.  Finally, Lemma~\ref{lemma:RGB-adjacent-blue} implies that we cannot simultaneously have $\gamma_1$ meeting $\ell$ and $\gamma_2$ meeting $m$.  Thus either $\gamma_1$ meets $e$ or $\gamma_2$ meets $g$.  After mapping to the diagram, endpoints of $e$ and $g$ are separated from those of $f$ by the closed curve $a \cup c \cup i \cup j$.  Note that Lemma~\ref{lemma:blue-prime}\eqref{lmitm:distinct-regions-cross} implies $e$ cannot meet the region of $a$ (which is also the region of $c$), and $g$ cannot meet the region of $c$ (also the region of $a$).  Hence the only way for $\gamma_1$ to meet $e$ or $\gamma_2$ to meet $g$ is if $G$ runs over $i$ or $j$.  It does not meet $j$ by Lemma~\ref{lemma:no-green-meets-bcj}.  Hence $G$ meets $i$, as claimed.

Now, $G$ intersects $i$, so a green edge lies in the triangle $abi$. By Lemma~\ref{lemma:no-green-meets-bcj}, it does not run to $b$.  Thus it runs from $i$ to $a$.  Since $f$ and $a$ are in distinct regions, by assumption, $a$ and $f$ must meet this green disk on opposite sides, as in Figure~\ref{fig:af-oppsides}.

\begin{figure}
  \input{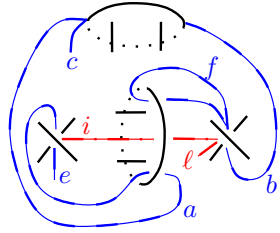}
  \caption{Edges $a$, $f$, and $i$ intersect some green crossing disk}
  \label{fig:af-oppsides}
\end{figure}

Now, edges $e$ and $\ell$ meet at a crossing.  The edge $e$ cannot be in the region of $a$ by Lemma~\ref{lemma:blue-prime}\eqref{lmitm:distinct-regions-cross}.
The edge $e$ cannot lie in the region containing $f$ because these are separated by $a \cup i\cup j \cup c$.  So $e$ cannot meet $G$.
Therefore, the only way for $\ell$ to meet the endpoint of $e$ is for $\ell$ to run through $G$.  Hence a green edge runs through the square $eif\ell$ with one endpoint on $\ell$.  That edge cannot meet $i$, by Lemma~\ref{lemma:red-green-bigons}.  It cannot meet $e$ because $e$ is disjoint from $G$.  So it meets $f$ and continues into a green edge in the square $fmgj$.  This green edge cannot meet $j$, by Lemma~\ref{lemma:no-green-meets-bcj}.  It cannot meet $m$ by Lemma~\ref{lemma:RGB-adjacent-blue}.  It cannot meet $g$, or $a$ and $g$,
therefore $c$ and $g$, are in the same region, contradicting Lemma~\ref{lemma:blue-prime}\eqref{lmitm:distinct-regions-cross}. This gives a
contradiction.
\end{proof}

\begin{lemma}\label{lemma:fg-regions}
If the edge labeled $f$ in Figure~\ref{fig:tri-squares-labels} is mapped to the same region as edges labeled $a$ and $c$, then the crossing at the endpoints of images of $b$ and $f$ is associated to the crossing circle corresponding to the common vertex of the triangles of Figure~\ref{fig:tri-squares-labels}, and in addition, at least one of the crossings at the endpoints of images of $h$ and $d$ or of $g$ and $c$ is associated to the same crossing circle.
\end{lemma}

Notice that Lemma~\ref{lemma:fg-regions} implies something symmetric for $g$ by flipping the adjacent triangle--square pairs upside down and relabeling: If $g$ is mapped to the region of images of $b$ and $d$, then the crossing at the endpoints of images of $c$ and $g$ is associated to the crossing circle of the vertex, in addition to at least one of the crossings at the endpoints of images of $e$ and $a$ or $b$ and $f$.

\begin{proof}
Again we refer to edges and their images by the same name.  Suppose that $f$ is in the region of $a$ and $c$.  Then Lemma~\ref{lemma:blue-prime}\eqref{lmitm:assoc-cross} implies that the crossing meeting endpoints of $f$ and $b$ is associated with the crossing circle $C$ corresponding to the vertex of the graph.

Suppose that the crossing at the endpoints of $d$ and $h$ is not associated to $C$.  This implies that $h$ cannot be in the region of $a$ and $c$, by Lemma~\ref{lemma:blue-prime}\eqref{lmitm:assoc-cross}.  In cases I and III of Figure~\ref{fig:trisquare-options}, $h$ is automatically either in the region of $f$ or in the region of $a$ and $c$, so these cases are impossible.  Thus we are either in case II or IV.

Notice that there are two versions of case II, where the crossings at the ends of $a$ and $c$ are the same, and where the crossings at the ends of $b$ and $d$ are the same.  In the latter case, we deduce that the crossing at the endpoints of $d$ and $h$ must be associated with $C$.  So we may assume in case II that the crossings at the ends of $a$ and $c$ are the same, as shown in Figure~\ref{fig:trisquare-options}.  

Now suppose that the crossing at the endpoints of $c$ and $g$ is not associated to the crossing circle.  Then $g$ cannot meet the region of $b$ and $d$ by Lemma~\ref{lemma:blue-prime}\eqref{lmitm:assoc-cross}.  By Lemma~\ref{lemma:fg-endpoints} it cannot have both endpoints on the same crossing.  But the endpoint of $g$ is contained in the region bounded by $b \cup j \cup k \cup d$.  Thus the other endpoint of $g$ must also be contained in that region, and must lie on a new crossing disjoint from the existing crossings in the diagram, as in Figure~\ref{fig:g-newcrossing}.

\begin{figure}
  \input{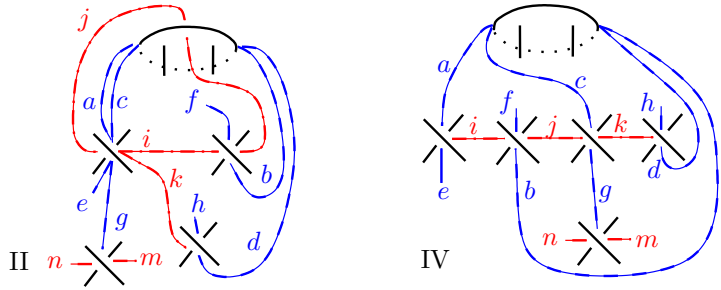}
  \caption{Edge $g$ meets a new crossing within the region bounded by $b\cup j\cup k
    \cup d$}
  \label{fig:g-newcrossing}
\end{figure}

Edges $n$ and $h$ have endpoints on the same crossing, but their other endpoints are bounded away from each other by $b\cup j\cup k\cup d$.  Because $h$ does not meet the region of $b$ or $d$, the edge $n$ must either meet $k$ or $j$.

Suppose first that $n$ meets $j$.  Then $n$ crosses into a region bounded away from the endpoint of $h$ by $a \cup c \cup i \cup j$.
We may assume, by applying a small homotopy which does not change the graphs $\Gamma_B$, $\Gamma_{BR}$, and $\Gamma_{BRG}$, that $n$ intersects $j$ at most once.  
Since $h$ does not meet the region of $a$ or $c$, the edge $n$ must then run through $i$.  But then $f$ is separated from the region of $a$ and $c$ by red edges, contradicting assumption.

So we conclude that $n$ meets the edge $k$.  Now, the edge $g$ along with a portion of $n$ and a portion of $k$ form a closed curve $\gamma$ meeting the diagram of $K_{B,i}$ twice with crossings on either side. Lemma~\ref{lemma:K2-0-prime} implies that we are in the diagram of $K_{B,0}$, and one of the red arcs, either $n$ or $k$, must meet a green disk.  Moreover, the closed curve $\gamma$ intersects this green disk transversely, exactly once.

Suppose first that $k$ meets a green disk $G$.  Then there is a green edge inside the triangle $cdk$, which must run from $k$ to $d$ by Lemma~\ref{lemma:no-green-meets-bcj}, hence $d$ also meets $G$.
Since the portions of $k \cup n$ making up $\gamma$ intersect $G$ only once, the diagram must be as in Figure~\ref{fig:n-meets-k-1}, left. (Note that $d$ cannot cross $g$ since they are in distinct regions.) Now consider the portion of $d$ shown in that figure, along with a portion of $k$ and an arc on $G$.  This forms a new closed curve $\gamma'$ meeting the diagram exactly twice, with crossings on either side.  If the portion of $\gamma'$ on the red meets another green disk, again the disk must meet $d$, so we may replace $\gamma'$ by a smaller closed curve consisting of an arc on the new green disk, and remnants of $\gamma'$, as in the proof of Lemma~\ref{lemma:red-blue-bigons}.  As in the proof of that lemma, by induction, we obtain a simple closed curve meeting the diagram of $K_{B,0}$ twice without meeting any crossing disks, and with crossings on both sides, contradicting Lemma~\ref{lemma:K2-0-prime}.

\begin{figure}
  \input{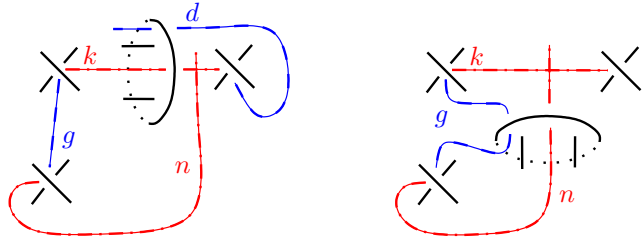}
  \caption{Left:  Crossing disk meeting $k$ and $d$.  Right:  Crossing
  disk meeting $n$ and $g$.}
  \label{fig:n-meets-k-1}
\end{figure}

Hence $n$ must meet a crossing disk $G$ between the endpoint of $n$ at $g$ and the point where $n$ meets $k$.  Thus there is a green edge with endpoint on $n$ in the square $gnhk$.  This edge cannot have its other endpoint on $k$, by Lemma~\ref{lemma:red-green-bigons}.  We claim it cannot have an endpoint on $h$.  This is because $h$ has both endpoints in the region bounded by $b\cup j\cup k\cup d$.  If the green edge happened to end on $h$, then the green disk must cross $b \cup j \cup k \cup d$.  It cannot cross $j$, by Lemma~\ref{lemma:no-green-meets-bcj}.  If it crosses $b$ or $d$, then either $g$ or $h$ must be in the region of $b$ and $d$, which is impossible by assumption, and by Lemma~\ref{lemma:blue-prime}\eqref{lmitm:distinct-regions-cross}.  If $G$ meets $k$, then there is a green arc in the triangle $cdk$, which again implies $G$ meets $d$, and so $g$ or $h$ is in the region of $b$ and $d$, which is impossible.  

Thus the green edge in square $gnhk$ with one endpoint on $n$ has its other endpoint on $g$.  Because the portion of $n \cup k$ making up $\gamma$ meets the crossing disk exactly once, the disk must be as shown in Figure~\ref{fig:n-meets-k-1}, right.  Again portions of $n$, $g$, and the crossing disk form a simple closed curve $\gamma'$ meeting the diagram of $K_{B,0}$ exactly twice.  Again, as in the proof of Lemma~\ref{lemma:red-blue-bigons}, replacing this closed curve if necessary, we obtain a contradiction to Lemma~\ref{lemma:K2-0-prime}.
\end{proof}

\begin{lemma}\label{lemma:tri-squares-assoc}
If there are three adjacent triangle--square pairs, as in Figure
\ref{fig:tri-squares-labels}, then for one of the following pairs of
crossings, both crossings are associated to the crossing circle
corresponding to the vertex:
\begin{enumerate}
\item crossings at the endpoints of $f$ and $b$ and of $d$ and $h$, or
\item crossings at the endpoints of $f$ and $b$ and of $g$ and $c$, or 
\item crossings at the endpoints of $g$ and $c$ and of $a$ and $e$.
\end{enumerate}
\end{lemma}

\begin{proof}
We prove the lemma by contradiction.  By Lemmas \ref{lemma:fg-regions} and \ref{lemma:fg-endpoints}, if there are three adjacent triangles and squares but the result does not hold, then edges $f$ and $g$ cannot be in the regions of edges $a$ and $c$, and $b$ and $d$, respectively, and neither $f$ nor $g$ can have endpoints on the same crossing.  Thus $f$ ends in a new crossing in the region $a \cup i \cup j \cup c$, and $g$ ends in a new crossing in the region $b \cup d \cup j \cup k$.

We may add these new crossings to Figure \ref{fig:trisquare-options}. However, first note that in options I and II from that figure, we have sketched the diagram so that $j$ runs through the crossing disk corresponding to the vertex of the graph.  In fact, if $j$ runs through that crossing disk, then a green edge corresponding to this crossing disk runs through the square $fmgj$, with one endpoint on $j$.  By Lemma \ref{lemma:red-green-bigons}, the other endpoint must be either on $f$ or $g$.  This implies that either $f$ meets the region of $a$ and $c$, or $g$ meets the region of $b$ and $d$, both of which we have ruled out.  Hence we need to adjust the diagrams of I and II so that $i$ meets the crossing disk corresponding to the vertex instead of $j$.  In III and IV, $i$ or $k$ might meet that crossing disk as well, but the same argument shows that $j$ cannot.  Updated sketches of the four possibilities are shown in Figure \ref{fig:fg-terminate}.

\begin{figure}
  \input{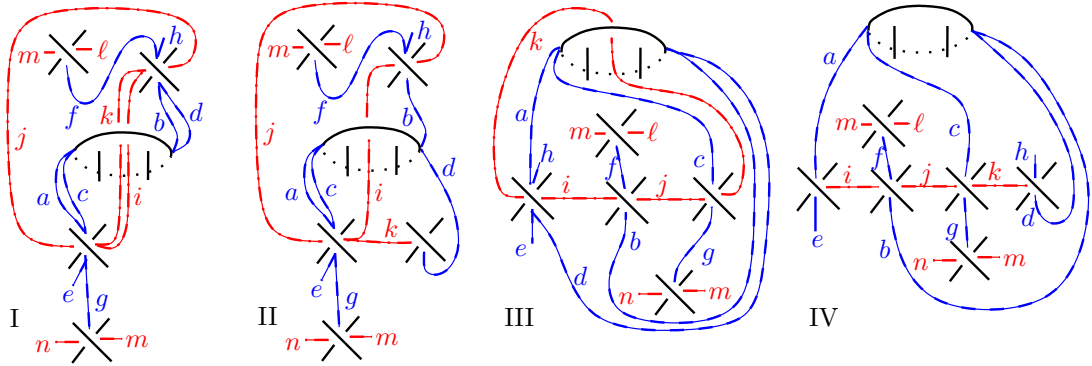}
  \caption{Four possibilities for endpoints of $f$ and $g$}
  \label{fig:fg-terminate}
\end{figure}

Note that the two endpoints of $m$ are shown in each figure.  In order for these endpoints to join, either $m$ must intersect $j$, or the crossing disk corresponding to the vertex.  This is because the endpoints of $m$ are separated by the closed curve $b \cup j \cup c \cup (\mbox{crossing disk})$.

Suppose first that $m$ intersects $j$.  Then $f$, along with a portion of $j$ and a portion of $m$, form a closed curve meeting the diagram of $K_{B,i}$ twice with crossings on either side.  Lemma~\ref{lemma:K2-0-prime} immediately implies that the diagram is that of $K_{B,0}$, not $K_{B,2}$.  Lemmas~\ref{lemma:K2-0-prime} and~\ref{lemma:no-green-meets-bcj} imply that the portion of $m$ meets some green crossing disk.  Then there is a green edge in the square $fmgj$ meeting either the edge $f$ or $g$. If $f$, then we may choose such a crossing disk closest to the crossing shared by $m$ and $f$, and argue as in the proof of the previous lemma that we obtain a contradiction to Lemma~\ref{lemma:K2-0-prime}.  So the green edge runs from $m$ to $g$.  But this green crossing disk must be contained in the region bounded by $i \cup j \cup a \cup c$, or have endpoints in the region of $a$ and $c$.  Since $g$ cannot meet the region of $a$ and $c$, we obtain a contradiction.

The only remaining possibility is that $m$ intersects the crossing disk corresponding to the vertex.  Then there is a corresponding green edge running through the square $fmgj$, with one endpoint on $m$.  Lemma~\ref{lemma:red-green-bigons} implies its other endpoint is not on $j$.  Hence its endpoint is on $f$ or $g$.  But that is possible only if $f$ is in the region of $a$ and $c$, or $g$ is in the region of $b$ and $d$.
\end{proof}

\begin{lemma}\label{lemma:No-fg-assoc}
If there are three adjacent triangle--square pairs as in Figure \ref{fig:tri-squares-labels}, then it cannot be the case that the crossings at the endpoints of images of $b$ and $f$ and of $c$ and $g$ are both simultaneously associated to the crossing circle corresponding to the vertex.
\end{lemma}

\begin{proof}
By Lemma~\ref{lemma:distinct-crossings}, the crossing at the endpoints of $f$ and $b$ and the crossing at the endpoints of $g$ and $c$ must be distinct.  Suppose both are associated to the crossing circle corresponding to the vertex, call it $C$. Then there are two distinct crossings associated to $C$, so the diagram cannot be that of $L_{B,0}$ by definition.  Hence we may assume that the diagram is that of $L_{B,2}$ and no additional crossings are associated to $C$.

Consider each of the options from Figure~\ref{fig:trisquare-options}.
We will show none of these can occur.

First consider option I.  The diagram has two crossings shown, and by assumption both of those crossings are associated to the crossing circle shown.  This means the region of $a$ and $c$ is the region of $f$ and $h$, and the region of $b$ and $d$ is the region of $e$ and $g$.  Then we can draw arcs in these regions from one crossing to another.  These arcs, along with the edges $i$ and $j$, will give closed curves in the diagram of $K_{B,2}$ meeting the diagram twice. Because the diagram of $K_{B,2}$ is prime, each bounds a strand of the knot and no crossings.  Thus the diagram of $K_{B,2}$ is that of a $(2,2)$--torus link, with a crossing circle of $L_{B,2}$ encircling the two crossings.  This contradicts Lemma~\ref{lemma:22torus}.

Now consider option II.  By assumption, the edge $f$ is in the region of $a$ and $c$, the edge $g$ is in the region of $b$ and $d$.  Draw an arc $\alpha$ inside the region of $b$ and $d$ from the crossing at the endpoint of $e$ (and $g$) to the crossing at the endpoint of $b$.  Now $\alpha \cup i$ meets the diagram of $K_{B,2}$ twice, with crossings on either side.  This contradicts Lemma \ref{lemma:K2-0-prime}.

For option III, note the crossing at endpoints of $a$ and $d$ must be associated to the crossing circle shown, by Lemma~\ref{lemma:blue-prime}\eqref{lmitm:assoc-cross}.  By assumption, the other two crossings are associated to that crossing circle as well. Thus three distinct crossings are associated to the crossing circle, contradicting the definition of $L_{B,2}$.

Finally, consider option IV.  Since the crossings at endpoints of $b$ and $f$ and of $c$ and $g$ are associated with the crossing circle shown, the crossing disk it bounds intersects the edge labeled $j$, by our convention on the diagram of Definition~\ref{def:L0-L2}.  Now consider the green edges of intersection in the graph.  One green edge runs from the vertex to $j$.  From there, a green edge must run to $f$ or $g$, but not $m$ by Lemma~\ref{lemma:red-green-bigons}.  The cases $f$ and $g$ are symmetric, so suppose it runs to $f$.  Now there must be a green edge in the square $eif\ell$.  It cannot run from $f$ to $i$, by Lemma \ref{lemma:RGB-adjacent-blue}.  It cannot run from $f$ to $e$, or $e$ would be in the region of $b$ and $d$, and Lemma~\ref{lemma:blue-prime}\eqref{lmitm:assoc-cross} would imply that a third crossing is associated to the crossing circle shown, which is a contradiction.  Thus the green edge runs from $f$ to $\ell$, and $\ell$ must be parallel to $j$ in the red bigon region containing $j$.  But then $e$ meets one of the crossings at the endpoints of $j$, either at the endpoint of $b$ or at the endpoint of $c$.  If $e$ meets $b$, then $e$ belongs to the region of $b$, and again we have too many crossings associated to the crossing circle.  If $e$ meets $c$, we get a contradiction to Lemma~\ref{lemma:blue-prime}\eqref{lmitm:distinct-regions}.
\end{proof}

\begin{lemma}\label{lemma:ae-assoc}
In three adjacent triangle--square pairs as in Figure~\ref{fig:tri-squares-labels}, if the crossings at the endpoints of $f$ and $b$ and of $h$ and $d$ are both associated to the crossing circle of the vertex, then so is the crossing at the endpoint of $a$ and $e$.
\end{lemma}

\begin{proof}
Suppose first that the crossings at the endpoints of $f$ and $b$ and at the endpoints of $h$ and $d$ are actually the same crossing.  If that crossing is associated to the crossing circle of the vertex, call it $C$, then Lemma~\ref{lemma:No-fg-assoc} implies that the crossing at the endpoints of $g$ and $c$ is not associated to $C$.  Assume by way of contradiction that the crossing at the endpoints of $a$ and $e$ is also not associated $C$.  Then Lemma~\ref{lemma:blue-prime}\eqref{lmitm:assoc-cross} implies that neither $g$ nor $e$ can meet a region meeting $C$, hence both $g$ and $e$ are disjoint from the crossing disk bounded by $C$.

Since $C$ is associated to the crossing at endpoints of $f$ and $b$, we know that the corresponding crossing disk intersects either $i$ or $j$, by our convention on diagrams (Definition~\ref{def:L0-L2}).  The arguments for both are nearly identical (in fact, symmetric in the squares $fjgm$ and $ief\ell$). We will walk through the argument for $j$ and leave the case of $i$ to the reader.  So suppose there is a green edge with one endpoint on $j$ running through the square $fjgm$.  The other endpoint cannot lie on $m$, by Lemma~\ref{lemma:red-green-bigons}.  It cannot lie on $g$ by our above observation that $g$ is disjoint from this crossing disk.  Hence the green edge runs from $j$ to $f$.  See Figure~\ref{fig:bd-sameassoc}.

\begin{figure}
  \input{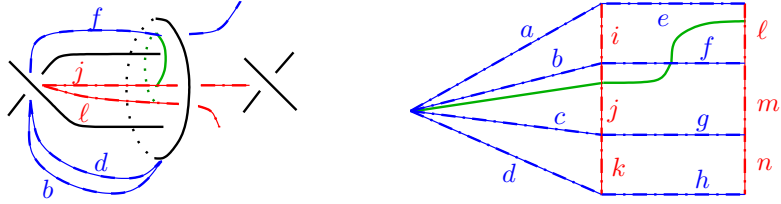}
  \caption{Left: Form of diagram if crossing at endpoints of $b$ and $d$ agree, and is associated to the crossing circle of the vertex.}
  \label{fig:bd-sameassoc}
\end{figure}

But then there is a green edge in the square $eif\ell$ with one endpoint on $f$.  It cannot have its other endpoint on $i$ by Lemma~\ref{lemma:RGB-adjacent-blue}.  It cannot have its other endpoint on $e$ since $e$ is disjoint from the region of $C$.  Thus its final endpoint is on $\ell$, and so $\ell$ runs through the crossing disk. But then $\ell$ has an endpoint on the crossing at endpoints of $b$, $j$, and $f$.  Because the square $eif\ell$ and the triangle $bcj$ are both mapped either above or below the projection plane (in Figure~\ref{fig:bd-sameassoc}, they are shown mapped below), Lemma~\ref{lemma:edge-intersections}\eqref{lmitm:blue-red} implies that $j$ and $b$ meet in the same regions of the diagram as $\ell$ and a blue edge meeting $\ell$ in a vertex.  So either $e$ or $f$ is the same region as $b$.  The edge $f$ cannot be in the region of $b$ by Lemma~\ref{lemma:blue-prime}\eqref{lmitm:distinct-regions-cross}.  But the edge $e$ cannot be in the region of $b$ either, because $e$ is not in a region meeting $C$.

This contradiction implies that if both the crossing at the endpoints of $f$ and $b$ and the crossing at the endpoints of $h$ and $d$ are associated with $C$, then they must be distinct crossings. So, the diagram cannot be that of $L_{B,0}$.

But now notice that if two crossings are associated to the same crossing circle, then they come from the same twist region, so there is a (red) bigon between them.  Then either $j$ lies in that bigon region or $i$ does.  If $j$ lies in the bigon, then its two endpoints must lie on the two crossings of the bigon.  Hence $c$ has an endpoint on one of those two crossings, and it is associated with the crossing circle of the vertex.  This contradicts Lemma \ref{lemma:No-fg-assoc}.  Thus $i$ lies in the bigon region between the two crossings.  It follows that $a$ has its endpoint on the crossing meeting $d$.  Thus the crossing at the endpoints of $a$ and $e$ is the same as the crossing at the endpoint of $d$ and $h$, and so it must be associated to the crossing circle corresponding to the vertex.  
\end{proof}

\begin{lemma}\label{lemma:L0-no-3trisquares}
There cannot be three adjacent triangle--square pairs for the diagram of $L_{B,0}$.  
\end{lemma}

\begin{proof}
Lemma \ref{lemma:tri-squares-assoc} implies that images of edges from the triangle--square pairs must meet at two crossings associated to the crossing circle $C$ corresponding to the vertex.  In $L_{B,0}$, at most one crossing can be associated to any crossing circle.  So the two crossings from Lemma \ref{lemma:tri-squares-assoc} must actually be the same crossing of the diagram.  Lemma \ref{lemma:distinct-crossings} implies that endpoints of $b$ and of $c$ map to distinct crossings, so those two cannot map to crossings associated to $C$.  Lemma \ref{lemma:ae-assoc} implies that if endpoints of $b$ and $d$ map to a crossing associated to $C$, then so does $a$, and again Lemma \ref{lemma:distinct-crossings} implies we have two crossings associated to $C$, contradicting the definition of $L_{B,0}$.  The only remaining possibility is that endpoints of $c$ and of $a$ map to a crossing associated to $C$.  But then by relabeling, we may again apply Lemma \ref{lemma:ae-assoc} to obtain a distinct crossing associated to $C$.  In all cases, we have a contradiction.
\end{proof}

\begin{lemma}\label{lemma:L2-no-5trisquares}
There cannot be five adjacent triangle--square pairs for the diagram of $L_{B,2}$.  
\end{lemma}

\begin{proof}
In three adjacent triangle--square pairs, there are four points where red edges meet blue edges to form triangles.  Lemmas \ref{lemma:tri-squares-assoc}, \ref{lemma:No-fg-assoc}, and \ref{lemma:ae-assoc} imply that three of these four points must map to crossings associated to the crossing circle corresponding to the blue vertex, and moreover, two of those four are adjacent to the top (or bottom by symmetry), and the last is adjacent to the bottom (resp.\ top).  Start with three such triangle--square pairs, and without loss of generality assume that the top two points are associated to the crossing circle.  This is shown on the left of Figure \ref{fig:5trisquare}.

\begin{figure}
  \includegraphics{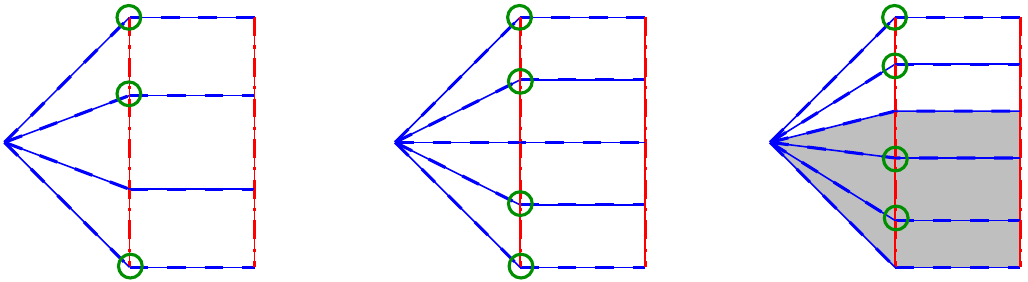}
  \caption{Green circles correspond to points which map to crossings     associated with the crossing circle of the vertex.  Shown are     three, four, and five adjacent triangle--square pairs.  Note that     in the case of five, the squares shown in gray contradict Lemma~\ref{lemma:No-fg-assoc}.  }
  \label{fig:5trisquare}
\end{figure}

Now attempt to add a fourth adjacent triangle--square pair.  This new pair will form another set of three adjacent triangle--square pairs. It cannot be added to the top of the three, else the new set of three adjacent triangle--square pairs on the top will contradict Lemma~\ref{lemma:No-fg-assoc}.  Hence it must be added to the bottom, and the new point on the new triangle where blue meets red must be associated to the crossing circle of the blue vertex.  This is shown in the center of Figure~\ref{fig:5trisquare}.

Finally, attempt to add a fifth adjacent triangle--square pair.  This cannot be added to top or bottom, or we contradict Lemma \ref{lemma:No-fg-assoc}.  Thus there cannot be five adjacent triangle--square pairs.  
\end{proof}

\begin{prop}\label{prop:L0-no-3adjbigons}
When $i=0$, the graph $\Gamma_B$ cannot contain three or more adjacent non-trivial bigons.
\end{prop}

\begin{proof}
Suppose the graph $\Gamma_B$ contains three adjacent non-trivial bigons.  By Lemma~\ref{lemma:no-blue-bigons}, these bigons cannot be disjoint from the red and green surfaces. The bigons cannot be disjoint from the red surface, for the following reason. The bigons would have to contain at least one green edge, by Lemma \ref{lemma:no-blue-bigons} (no non-trivial blue bigons). If there is a green edge which intersects the interior of a blue edge, this contradicts Lemma \ref{lemma:blue-green-bigons} (no blue-green bigons) or \ref{lemma:noBBG} (no blue-blue-green triangles). On the other hand, if both endpoints of the green edge are the same vertex of $\Gamma_B$, this contradicts Lemma \ref{lemma:NoGreenLoop} (no green monogons). If the endpoints of the green edge are distinct vertices of $\Gamma_B$, this implies that the edges of the blue bigons are trivial, which is contrary to hypothesis. So, the blue bigons intersect the red surface. Lemma~\ref{lemma:red-blue-bigons} (no red-blue bigons) implies any intersection with the red surface must run straight through all three adjacent bigons.  An outermost such intersection cuts off three adjacent blue--blue--red triangles.

By Lemma~\ref{lemma:adjacent-BBR-triples}, there cannot be a single such intersection of red, cutting off two pairs of three adjacent blue--blue--red triangles adjacent across the red edges.  Hence there are at least two such intersections of red, and the outermost two cut off three adjacent red triangles and three adjacent red squares.  But now Lemma~\ref{lemma:L0-no-3trisquares} implies that this is impossible when $i=0$.  
\end{proof}

\begin{prop}\label{prop:L2-no-5adjbigons}
When $i=2$, the graph $\Gamma_B$ cannot contain five or more adjacent non-trivial bigons.
\end{prop}

\begin{proof}
Suppose the graph contains five adjacent non-trivial bigons.  As in the proof of Proposition~\ref{prop:L0-no-3adjbigons}, we argue that the bigons cannot be disjoint green and red, by Lemma~\ref{lemma:no-blue-bigons}, cannot be disjoint red, by Lemmas~\ref{lemma:blue-green-bigons}, \ref{lemma:NoGreenLoop} and~\ref{lemma:noBBG}, and red intersections run through each bigon, splitting off triangles and squares by Lemma~\ref{lemma:red-blue-bigons}.  Again Lemma~\ref{lemma:adjacent-BBR-triples} implies there is more than one intersection of red, and the two outermost intersections cut off five adjacent triangles adjacent five adjacent squares across red edges. This contradicts Lemma~\ref{lemma:L2-no-5trisquares}.
\end{proof}

This completes the proof of Theorem \ref{thm:Bincompressible}.

\section{Boundary--$\pi_1$--injectivity}\label{sec:boundary}
In this section, we finish proofs of lemmas needed to show twisted surfaces are boundary--$\pi_1$--injective.  To do so, we analyze further the graph $\Gamma_B$ from Lemma~\ref{lemma:graph-bdyincompr}.  In the graph from that lemma, blue edges may have at least one endpoint on the part of $\partial D$ which maps to $\partial N(K)$.  Where two blue edges leave the same high valence vertex and both end on $\partial N(K)$, we obtain a triangle with two blue sides and one side on $\partial N(K)$.  We show in this section that we cannot have many adjacent triangles of this form, by restricting the graph $\Gamma_B$.

\begin{lemma}\label{lemma:adjtri-nored}
In $L_{B,0}$, there cannot be two adjacent triangles of $\Gamma_{B}$ each with two blue sides and one side on $\partial N(K)$, unless the triangles meet the red surface.

Similarly, in $L_{B,2}$, there cannot be three adjacent triangles of $\Gamma_{B}$ with two blue sides and one on $\partial N(K)$, unless the triangles meet the red surface.
\end{lemma}

\begin{proof}
Suppose two adjacent triangles in $D$ do not meet the red surface. Label the three blue edges of the triangle as in Figure~\ref{fig:triangle-label}, left.  Without loss of generality, we may assume that the triangle with edges labeled $a$ and $b$ is mapped above the plane of projection, and that the one with edges labeled $b$ and $c$ is mapped below. Since edges $a$ and $b$ start on opposite sides of the same crossing circle, they must end in different regions of the diagram, by Lemma~\ref{lemma:blue-prime}\eqref{lmitm:distinct-regions}.  Because the third side of the triangle they form lies on the link, that third side must run over a single over--crossing of the diagram.  This is shown in Figure~\ref{fig:triangle-label}, second from left.  Similarly, the triangle with sides $b$ and $c$ must have third side running over a single under--crossing of the diagram.  Putting these together, the result must be as in Figure~\ref{fig:triangle-label}, third from left.  Note each of these crossings is associated to the crossing circle corresponding to the blue vertex of the triangles, by Lemma~\ref{lemma:blue-prime}\eqref{lmitm:assoc-cross}.  Hence the crossing circle has at least two crossings which are associated to it.  This contradicts our construction of $L_{B,0}$: at most one crossing belongs to any crossing circle.

\begin{figure}
  \input{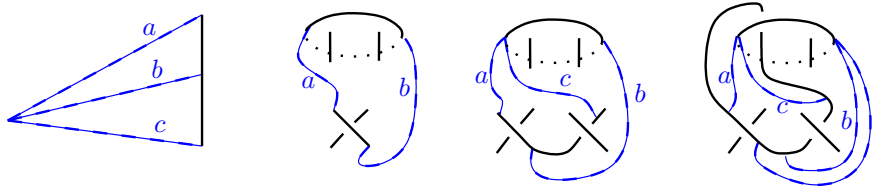}
  \caption{Left to right: Labels on edges of three adjacent triangles. A triangle above the plane of projection. Two adjacent triangles. Three adjacent triangles meeting only two crossings.}
  \label{fig:triangle-label}
\end{figure}

If a third triangle is adjacent, then either a third crossing will be associated to the crossing circle, which is impossible for $L_{B,2}$, or the crossing straddled by endpoints of edges in the third triangle will agree with one of the existing crossings, as on the right of Figure~\ref{fig:triangle-label}.  However, repeated applications of primality, Lemma~\ref{lemma:K2-0-prime}, implies that in this case the diagram is that of a $(2,2)$--torus link encircled by a crossing circle.  This contradicts Lemma~\ref{lemma:22torus}.  
\end{proof}

\begin{lemma}\label{lemma:adjtri-onered}
In the graph $\Gamma_B$, three adjacent triangles, each with two blue sides and one side on $\partial N(K)$, must meet the red surface more than once.
\end{lemma}

\begin{proof}
Suppose not.  By Lemma~\ref{lemma:adjtri-nored}, any three adjacent triangles must meet the red surface at least once.  Suppose three adjacent triangles meet the red surface only once.  Because there are no red--blue bigons (Lemma~\ref{lemma:red-blue-bigons}), the red must run across the three triangles, meeting each triangle in both of its blue edges.  Label the edges of adjacent triangles meeting the red once as in Figure~\ref{fig:adjtri-onered}.  By Lemma~\ref{lemma:distinct-crossings}, each of the blue--blue--red triangles meets two distinct crossings.  As before, the three triangles can meet two, three, or four crossings, and the four possibilities are sketched in Figure~\ref{fig:trisquare-options}.  However, now $e$ and $f$ must straddle an undercrossing, $f$ and $g$ straddle an overcrossing, and $g$ and $h$ straddle an undercrossing.

\begin{figure}
  \input{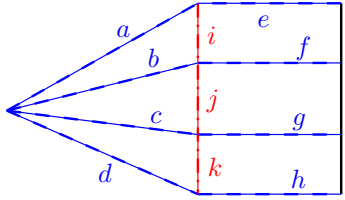}
  \caption{Labels on triangles.}
  \label{fig:adjtri-onered}
\end{figure}

We step through the cases in Figure~\ref{fig:trisquare-options} one by one, ruling out each case.

Case I.  The three blue--blue--red triangles meet just two crossings. In this case, $i\cup j$ separates endpoints of $e$ and $f$.  Since $e$ and $f$ straddle a crossing, but cannot meet $i$ or $j$, they must straddle one of the crossings between $i$ and $j$.  Similarly, $f$ and $g$ straddle the other crossing between $i$ and $j$.  Then both crossings shown in case I of Figure~\ref{fig:trisquare-options} belong to the crossing circle shown, and a primeness argument shows that the diagram is a $(2,2)$--torus link encircled by a crossing circle.  This violates Lemma~\ref{lemma:22torus}.

Case II.  Again $i\cup j$ separates endpoints of $e$ and $f$ and of $f$ and $g$, so again $e$ and $f$ straddle one of the crossings between $i$ and $j$, and $f$ and $g$ straddle the other, and both crossings are associated to the crossing circle shown.   So the diagram is that of $L_{B,2}$. Then one of $e$ or $g$ runs from the crossing at the endpoint of $a$ to the crossing at the endpoint of $b$.  Taking this edge, along with the edge $i$, gives a closed curve meeting the diagram of of $K_{B,2}$ twice with crossings on either side.  This contradicts Lemma~\ref{lemma:K2-0-prime}.

Case III.  This time, $i\cup j\cup k$ separates endpoints of $e$ and $f$, and of $f$ and $g$, and of $g$ and $h$.  Hence these pairs of edges straddle crossings between $i$, $j$, and $k$.  There are three such pairs and three such crossings.  In all cases, it can be shown that the pairs of edges straddle distinct crossings, and all three crossings are associated to the crossing circle shown.  This contradicts the definition of $L_{B,0}$ and $L_{B,2}$:  at most $1$ or $2$ crossings, respectively, can be associated to a given crossing circle.

Case IV.  In this case, $f$ and $e$ are not required to straddle one of the four crossings shown.  However, if they do not, then since endpoints of $f$ and $e$ are separated by $a\cup i\cup j\cup c$, and since $e$ cannot meet any of those curves, $f$ must intersect $a$ or $c$.  Thus the crossing straddled by $f$ and $e$ meets the same blue regions on either side as the crossing at the endpoint of $a$.  By the fact that the diagram of $K_{B,i}$ is blue twist reduced, Lemma~\ref{lemma:blue-twist-reduced}, these two crossings must bound a sequence of (at least one) bigon regions of the diagram between them. Then $f$ and $g$ straddle the next crossing in the bigon sequence, hence $e$ and $g$ are in the same region of the diagram.  But endpoints of $e$ and $g$ are bounded away from each other by $b\cup j\cup k\cup d$, hence $e$ and $g$ are both in the region of $b$ and $d$.  That implies that crossings at the endpoints of $a$ and $c$ belong to the crossing circle shown.  Since $f$ is in the region of $a$, the crossing at the endpoint of $b$ also belongs to that crossing circle.  This contradicts the definition of $L_{B,i}$.

Finally, it remains to show that $e$ and $f$ cannot straddle any of the crossings shown in the diagram on the left of Figure~\ref{fig:trisquare-options}.  Arguments similar to those above imply that for each of these crossings, if $e$ and $f$ straddle the crossing then three of the four crossings shown must be associated to the crossing circle shown.  We leave the details to the reader.

\end{proof}

\begin{lemma}\label{lemma:no-3BBKtriangles}
In the graph $\Gamma_B$, when $i=0$, there cannot be three adjacent triangles each with one side on $\partial N(K)$.  When $i=2$, there cannot be five such triangles.
\end{lemma}

\begin{proof}
Lemma~\ref{lemma:adjtri-nored} implies the triangles with two blue edges and one edge on $\partial N(K)$ must meet the red surface.  Lemma~\ref{lemma:adjtri-onered} implies that the triangles must meet the red surface at least twice.  But then the two intersections closest to the vertex of the triangles cut the triangles into blue--blue--red triangles adjacent to red and blue squares.  Lemma~\ref{lemma:L0-no-3trisquares} gives a contradiction in case $i=0$:  no disk can be mapped into $L_{B,0}$ such that three adjacent blue--blue--red triangles meet adjacent red and blue squares.  In the case $i=2$, Lemma~\ref{lemma:L2-no-5trisquares} gives a contradiction.
\end{proof}

This completes the proof of Theorem \ref{thm:bdryincompressible}.

\section{Properties of twisted surfaces}\label{sec:homotopic}
In this section, we investigate homotopy classes of arcs in twisted checkerboard surfaces.  This requires machinery of the previous sections, and has applications in \cite{lp:acv}.

Consider two arcs that are distinct and essential in the surface $S_{B,i}$, for $i=0, 2$, but homotopic when mapped into $S^3\setminus K$.  We determine when this can happen.

First, we introduce terminology.  A small regular neighborhood of a twist region in $S^3$ is a 3--ball which intersects both checkerboard surfaces of $K$.  We say that the intersection of the ball and a checkerboard surface is the subsurface \emph{associated} with the twist region.  If the twist region has $c$ crossings, then the intersection with one checkerboard surface is a disk, and the other has Euler characteristic $2-c$.  See Figure~\ref{fig:twist-assocsfce}.

\begin{figure}
  \includegraphics[width=2in]{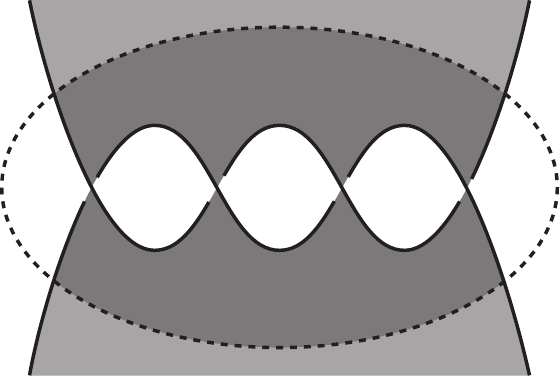}
  \caption{The subsurface of a checkerboard surface associated with a
    twist region.}
  \label{fig:twist-assocsfce}
\end{figure}

In the case of twisted checkerboard surfaces, we can make a similar definition. Consider a twist region of $K_i$. The checkerboard surfaces for $K_i$ have subsurfaces associated with this twist region. If the twist region is encircled by a crossing circle of $L_i$, we choose the subsurface in $R_i$ or $B_i$ so that it is punctured twice by this crossing circle. However, we arrange that the subsurfaces are disjoint from all other crossing circles. Since there are inclusions $R_i \subset S_{R,i}$ and $B_i \subset S_{B,i}$, we obtain subsurfaces of $S_{R,i}$ and $S_{B,i}$ which are \emph{the subsurfaces associated with the twist region of $K_i$}.

\begin{theorem}\label{thm:htpcarcs}
Suppose $K$ is a link with prime, twist reduced, alternating diagram (which we also call $K$).  Let $K_i$, $i=0, 2$, be the diagram obtained from that of $K$ by removing pairs of crossings from each twist region of $K$ with at least $N_\tw$ crossings, where $N_{\tw}\geq 72$ if $i=0$, and $N_{\tw}\geq 121$ if $i=2$, so that the diagram of $K_i$ has one or $i$ remaining crossings in any such twist region. Finally, suppose that two distinct essential arcs in the surface $S_{B,i}$ have homotopic images in $S^3\setminus{K}$, but are not homotopic in $S_{B,i}$. Then the two arcs are homotopic in $S_{B,i}$ into the same subsurface associated with some twist region of $K_i$.
\end{theorem}

We will prove Theorem~\ref{thm:htpcarcs} in a sequence of lemmas.  The first is an analogue of Lemmas~\ref{lemma:graph-bluevalence} and~\ref{lemma:graph-bdyincompr}.

\begin{lemma}\label{lemma:htpc-graph}
Suppose homotopically distinct essential arcs $a_1$ and $a_2$ in $S_{B,i}$ map by $f\co S_{B,i} \to S^3\setminus K$ to homotopic arcs $e_1$ and $e_2$ in $S^3\setminus{K}$.  Then there is a map of a disk $\phi\co D \to S^3\setminus {\rm int}(N(K))$ with $\partial D$ expressed as four arcs, with opposite arcs mapping by $\phi$ to $e_1$ and $e_2$, and the other two arcs mapping to $\partial N(K)$.

Moreover, $\Gamma_B=\phi^{-1}(f(S_{B,i}))$ is a collection of embedded closed curves and arcs and an embedded graph on $D$ whose edges have endpoints either at vertices where $\phi(D)$ meets a crossing circle, or on $\phi^{-1}(\partial N(K))$ on $\partial D$.  Each vertex in the interior of $D$ has valence a non-zero multiple of $2n_j$, where $2n_j$ is the number of crossings removed from the twist region at the relevant crossing circle.  Each vertex in the interior of an arc in $\partial D$ that maps to $S_{B,i}$ has valence $n_j+1$.  Each vertex on an arc in $\partial D$ that maps to $\partial N(K)$ has valence one.
\end{lemma}

\begin{proof}
The homotopy between $e_1$ and $e_2$ gives a map of a disk $\phi\co D \to S^3\setminus {\rm int}(N(K))$, with $\partial D$ mapped to the four arcs required by the lemma.  As in Lemmas~\ref{lemma:graph-bluevalence} and~\ref{lemma:graph-bdyincompr}, we may modify $\phi$ in two stages, first near $\partial D$ and then in the interior, to obtain the map with the desired properties.  

The arcs $e_1$ and $e_2$ both lift to the orientable double cover of $S_{B,i}$, which is transversely orientable, and by pushing $e_1$ and $e_2$ in this transverse direction, we obtain the map $\phi$ in a neighborhood of these arcs.  Similarly, using the fact that $\partial N(K)$ is transversely orientable, we can extend the definition of $\phi$ over a collar neighborhood of $\partial D$.  Now extend $\phi$ over all of $D$, and make it transverse to all crossing circles, and transverse to $f(S_{B,i})$.  

Let $\Gamma_B= \phi^{-1}(f(S_{B,i}))$ on $D$.  Because $S_{B,i}$ is embedded in $S^3\setminus K$ except at crossing circles, $\Gamma_B$ consists of embedded closed curves, embedded arcs (edges) with endpoints corresponding to points of intersection of crossing circles (vertices), or with endpoints on $\partial N(K)$.

As in the proof of Lemmas~\ref{lemma:graph-bluevalence} and~\ref{lemma:graph-bdyincompr}, a vertex in the interior of $D$ corresponds to a transverse intersection of $\phi(D)$ with a crossing circle in $S^3\setminus K$.  Hence the vertex has valence $2n_j$.  For a vertex on an arc in $\bdy D$ that maps to $S_{B,i}$, a neighborhood of the vertex maps to half a meridian disk for a crossing circle, and so the vertex has valence $n_j+1$. At a vertex on an arc in $\bdy D$ that maps to $\partial N(K)$, the arc in $\bdy D$ is transverse to $S_{B,i}$, and so this vertex of $\Gamma_B$ has valence one.
\end{proof}

We view the disk $D$ of the previous lemma as a square with west side mapping to $e_1$, east side mapping to $e_2$, and north and south sides mapping to $\bdy{N(K)}$.  As before, we also have graphs $\Gamma_{BR}$ and $\Gamma_{BRG}$, and complexity as in Definition~\ref{def:complexity}, ordered lexicographically.
We will take our graph to make the complexity as small as possible.  All the results of Sections~\ref{sec:surfaces}, \ref{sec:trisquares}, and~\ref{sec:boundary} will apply to these graphs.

The following is an analogue of Lemma \ref{lemma:planar-graph}.

\begin{lemma} \label{lemma:small-valence-I-times-I}
Let $\Gamma$ be a graph in the disk $I \times I$. Suppose that $\Gamma$ includes $\partial I \times I$. 
Suppose also that $\Gamma$ contains no bigons and no monogons, and that its intersection with $I \times \partial I$
is a collection of valence one vertices. Suppose also that there are no triangular regions, with one edge on
$I \times \partial I$. Finally, suppose that $\Gamma$ contains at least one vertex that does not lie in $I \times \partial I$.
Then either there is some vertex in the interior of $I \times I$ with valence at most $6$,
or there is a vertex on $\partial I \times (I-\partial I)$ with valence at most $4$.
\end{lemma}

\begin{proof}
Double the disk $I \times I$ along the two arcs $I \times \partial I$ to obtain an annulus. Then double the annulus to obtain a torus. At both stages, double the graph, and thereby a obtain a graph $\Gamma^+$ in the torus. 
The 1-valent vertices of $\Gamma$ on $I \times \partial I$ become midpoints of edges of $\Gamma^+$.
This graph $\Gamma^+$ has no monogons and no bigons, by our assumptions about $\Gamma$. It is well known that a graph in the torus with no monogons and no bigons contains a vertex with valence at most $6$. The proof is analogous to the Euler characteristic calculation in Lemma \ref{lemma:graph-in-2-sphere}. This vertex restricts to the required vertex in $\Gamma$.
\end{proof}

\begin{lemma}\label{lemma:small-valence-IxI2}
Let $\Gamma$ be a connected graph in $I\times I$ that includes $\partial I\times I$, has no monogons, and that has intersection with $I\times \partial I$ consisting of a collection of valence one vertices. Suppose also that each interior vertex of $\Gamma$ has valence at least $R_\tw$ and each vertex on $\partial I\times(I-\partial I)$ has valence at least $(R_\tw/2)+1$. Then $\Gamma$ must have more than $(R_\tw/8)-1$ adjacent bigons, or more than $(R_\tw/8)-1$ adjacent triangles with one edge on $I\times(\partial I)$. 
\end{lemma}

\begin{proof}
Suppose not. Then every collection of adjacent bigons or triangles has at most $(R_\tw/8)$ edges. Collapse each family of adjacent bigons and triangles to a single edge, forming a graph $\overline{\Gamma}$. By Lemma~\ref{lemma:small-valence-I-times-I}, $\overline{\Gamma}$ contains a vertex in the interior of the disk with valence at most $6$, or one on $\partial I\times(I-\partial I)$ with valence at most $4$. In the former case, the vertex came from a vertex of $\Gamma$ with valence at most $6(R_\tw/8)$, which is less than $R_\tw$. In the latter case, the vertex came from a vertex of $\Gamma$ with valence at most $4(R_\tw/8)$, which is less than $(R_\tw/2)+1$. In both cases, we get a contradiction. 
\end{proof}

The next lemma is analogous to Lemmas~\ref{lemma:adj-bigons} and~\ref{lemma:adj-bigons-2}.

\begin{lemma}\label{lemma:adj-bigons-3}
Let $\Gamma_B$ be the graph in $D$ provided by Lemma~\ref{lemma:htpc-graph}.  Suppose that $\Gamma_B$ has no monogons, and there are no trivial edges of $\Gamma_B$ in $\partial D$.  Finally, suppose that $\Gamma_B$ contains at least one blue vertex, i.e.\ a vertex mapping to a crossing circle in $S^3\setminus K$.  Then either $\Gamma_B$ has more than $(R_\tw/24)-1$ adjacent non-trivial bigons, or there are more than $(R_\tw/24)-1$ adjacent triangles, each with one arc on $\phi^{-1}(\partial N(K))$.
\end{lemma}

\begin{proof}
We argue as in the proof of Lemmas~\ref{lemma:adj-bigons} and~\ref{lemma:adj-bigons-2}. 

We again need to deal with trivial arcs. Recall from Definition \ref{def:trivial-edge} that a \emph{trivial arc} is a blue arc of $\Gamma_B$ that is disjoint from the red and green edges, that has endpoints on the same crossing circle, but that does not form a loop in $\Gamma_B$. According to Lemma \ref{no-interior-trivial-arcs}, each trivial arc must have at least one endpoint on $\partial D$. Note also that, by definition, trivial arcs must end at crossing circles, and so their endpoints do not lie on the part of $\partial D$ that maps to $\partial N(K)$.

By Lemma \ref{all-but-one-edge-trivial}, if all but one of the edges of a region of $\Gamma_B$ are trivial, then the remaining edge is also trivial. Hence, if one edge of a bigon of $\Gamma_B$ is trivial, then so is the other. We call this a \emph{trivial bigon}. Any bigon that shares an edge with a trivial bigon is also trivial, and therefore trivial bigons patch together to form discs called \emph{trivial bigon families}. If more than one trivial bigon family is incident to an interior vertex, then we consider all the trivial bigon families incident to this vertex, and call it a \emph{trivial star}.

As in Lemmas~\ref{lemma:adj-bigons} and~\ref{lemma:adj-bigons-2}, we consider
outermost disks in the complement of the collection of trivial stars and trivial arcs with both endpoints on $\partial D$, as well as innermost disks bounded by edge loops or containing a connected component of $\Gamma_B$. If one of these disks intersects $\partial D$ in a single point, or in a single arc that does not meet $\partial N(K)$, we pass to this subdisk, and the argument proceeds exactly as in Lemma~\ref{lemma:adj-bigons}.

If one of these subdisks intersects $\partial N(K)$, then we can arrange that the subdisk intersects at most one of the arcs that maps to $\partial N(K)$. This subdisk may have exceptional vertices, just as in the proof of Lemma~\ref{lemma:adj-bigons}. We double this disk along the arc that maps to $\partial N(K)$, to obtain a new disk.
Double the graph $\Gamma_B$ to form a graph $\Gamma_B^+$. The valence one vertices of $\Gamma_B$ on $\phi^{-1}(\partial N(K))$ become midpoints of edges of $\Gamma_B^+$.
Hence, every vertex in the interior of the new disk has valence at least $R_\tw$, and every unexceptional vertex on the boundary has valence at least $(R_\tw/2) +1$.
There are two collections of exceptional vertices on the boundary, consisting of at most six vertices. Now apply the argument in Step~2 of Lemma~\ref{lemma:adj-bigons}, except at the final stage, collapse each of the two collections of exceptional vertices to a single exceptional vertex, to ensure no more than two exceptional vertices when finished. The result is a graph $\Gamma$ with properties (a) and (c) as before. Because we now allow at most two trivial bigons meeting a vertex to be collapsed, the argument for adjacent bigons in (b) of that proof must be adjusted: adjacent bigons in $\Gamma$ come from at most four collections of adjacent bigons in $\Gamma_B^+$, collapsing at most three triangles and squares, but no more because our disk was outermost. Now, continuing as in the proof of Lemma~\ref{lemma:adj-bigons}, we obtain a graph $\Gamma$ to which the hypotheses of Lemma~\ref{lemma:planar-graph2} apply, and $\Gamma$ has a collection of more than $(R_\tw/6)-1$ adjacent bigons.
All but at most three came from a non-trivial bigon of $\Gamma_B^+$. These are divided into at most four collections of adjacent non-trivial bigons of $\Gamma_B^+$. So $\Gamma_B^+$ has more than $(R_\tw/24) -1$ adjacent non-trivial bigons. Hence $\Gamma_B$ has more than $(R_\tw/24)-1$ adjacent non-trivial bigons, or more than $(R_\tw/24)-1$ adjacent triangles, each with one edge mapping to $\partial N(K)$. 

So suppose now that $D$ is connected with no edge loops, contains no trivial arcs with both endpoints on $\partial D$, and no trivial stars. Then we do not pass to a subdisk of $D$. Instead, we apply the procedure given in Step~2 of the proof of Lemma~\ref{lemma:adj-bigons}, where trivial bigon families are removed, and replaced by vertices on $\partial D$, producing a connected graph $\Gamma$ satisfying properties (a), (b), and (c) as before. That is, $\Gamma$ has no monogons, its bigons come from collections of adjacent non-trivial bigons of $\Gamma_B$ plus no more than one square or two triangles of $\Gamma_B$, grouping at most three collections of non-trivial adjacent bigons of $\Gamma_B$ into adjacent bigons in $\Gamma$. Finally, every interior vertex of $\Gamma$ has valence at least $R_\tw$ and every boundary vertex disjoint from $\phi^{-1}(\partial N(K))$ has valence at least $(R_\tw/2)+1$. In this case, note there are no exceptional vertices because we did not pass to a subdisk. 

Now the hypotheses of Lemma~\ref{lemma:small-valence-IxI2} apply to $\Gamma$. Thus $\Gamma$ has a collection of more than $(R_\tw/8)-1$ adjacent bigons or triangles with an edge on $I\times(\partial I)$. All but at most two of these came from a non-trivial bigon of $\Gamma_B$, and these are divided into at most three collections of adjacent non-trivial bigons of $\Gamma_B$. So $\Gamma_B$ has more than $(R_\tw/24)-1$ adjacent non-trivial bigons. 
\end{proof}


\begin{lemma}\label{lemma:htpc-novertices}
If $i=0$ and $N_\tw \geq 72$, then the graph $\Gamma_B$ of Lemma~\ref{lemma:htpc-graph} contains no blue vertices.  Similarly, if $i=2$ and $N_\tw \geq 121$, then the graph $\Gamma_B$ contains no blue vertices.  
That is, $\phi(D)$ meets no crossing circles.
\end{lemma}

\begin{proof}
Lemma~\ref{lemma:no-monogons} implies that $\Gamma_B$ contains no monogons.  Lemma~\ref{lemma:no-boundary-trivial-arcs} implies that it contains no trivial arc that is a subset of $\partial D$.  So if $\Gamma_B$ contains a blue vertex, then Lemma~\ref{lemma:adj-bigons-3} applies.  If $N_\tw \geq 72$ and $i=0$, then by equation~\eqref{eqn:Rtw}, $R_\tw \geq 72$, so $\Gamma_B$ must contain more than two adjacent non-trivial bigons or adjacent triangles.  Similarly, if $i=2$ and $N_\tw\geq 121$, then $R_\tw\geq 120$, and $\Gamma_B$ must contain more than four adjacent non-trivial bigons or adjacent triangles.  This contradicts either Proposition~\ref{prop:L0-no-3adjbigons} or~\ref{prop:L2-no-5adjbigons}, or Lemma~\ref{lemma:no-3BBKtriangles}.  Hence there can be no blue vertices on $\Gamma_B$.
\end{proof}

As before, consider the red surface $R_i$ embedded in $S^3\setminus L_{B,i}$, and the graph $\Gamma_{BR}$, with $\Gamma_B$ coming from Lemma~\ref{lemma:htpc-graph}.  

\begin{lemma}\label{lemma:htpc-straightarcs}
The graph $\Gamma_B$ consists only of arcs whose two endpoints are on north and south sides of $\partial D$.  The red edges of the graph $\Gamma_{BR}$ consist of arcs with endpoints on distinct sides of $D$ (north, south, east, west).
\end{lemma}

\begin{proof}
By Lemma~\ref{lemma:htpc-novertices}, we may assume that $\Gamma_{BR}$
consists of red and blue arcs and simple closed curves, but no blue vertices. 
By Lemma~\ref{lemma:scc}, we may assume there are no blue simple closed curves.  Because there are no blue vertices, no blue edges can have an endpoint on the east or west sides (mapping to $e_1$ and $e_2$).  By Theorem~\ref{thm:bdryincompressible}, any blue arc with both endpoints on $\partial N(K)$ on the north (or both on the south) is trivial in the blue surface, and so we may replace $D'$ with a disk that does not meet that blue arc.
This will reduce the number of vertices of $\Gamma_B$, contradicting our assumption that the graph has minimum complexity.  Thus blue arcs run from north to south.  

As for the red, a red simple closed curve disjoint from blue bounds a disk on red, so $D'$ could be replaced by a disk that does not meet this red curve.  A red simple closed curve that is not disjoint from the blue would imply the existence of a red--blue bigon, contradicting Lemma~\ref{lemma:red-blue-bigons}.  Hence we may assume there are no red simple closed curves, whether or not they meet blue.  Because there are no red--blue bigons by Lemma~\ref{lemma:red-blue-bigons}, no red edge of $\Gamma$ can have both endpoints on the east or west side of $D$.  Finally, the red surface $R_i$ is a checkerboard surface for $K_i$, hence is boundary incompressible.
Thus any red arc with both endpoints on $\partial N(K)$ on the north (or both on the south) is trivial in the red surface, and so we may replace $D'$ with a disk that does not meet that red arc, without increasing the number of vertices of $\Gamma_B$.  This move contradicts the assumption that complexity is as small as possible.  The result follows.
\end{proof}

\begin{lemma}\label{lemma:htpc-opparcs}
The graph $\Gamma_{BR}$ consists of red and blue arcs with endpoints on opposite sides of $D$ (north--south, or east--west).
\end{lemma}

\begin{proof}
By the preceding lemma, blue arcs run north to south, as desired.  Hence we need to show there are no red arcs running from north to east, north to west, south to east, or south to west.  Any such red arc would have an endpoint on the east or west side mapping to the arc $e_1$ or $e_2$.  It may meet other blue arcs, running north to south, but in any case, the arc cuts off a triangle with one side on blue, one side on red, and one side on $\partial N(K)$.  By considering the region of $D$ where the arc meets the north or south side, we may take such a triangle to have interior disjoint from all other red and blue edges.

Because the triangle is disjoint from all crossing circles, we may sketch its image into the diagram of $L_{B,i}$, assuming without loss of generality that the triangle maps into the region above the plane of projection.  The blue arc has one endpoint on a strand of the link and one endpoint at a crossing.  The red has one endpoint on the same strand of the link (i.e.\ the portion of $K_i$ running between two adjacent under--crossings), and the other on the same crossing.  There are two ways that a red and a blue endpoint can meet the same strand of the link between under--crossings.  One way is if they are on the same side of that strand, but with that strand running over a crossing between them, as in Figure~\ref{fig:redblue-corner}, left.  The other is if they are on opposite sides of a strand that does not run over a crossing, as in Figure~\ref{fig:redblue-corner}, right.

\begin{figure}
  \includegraphics{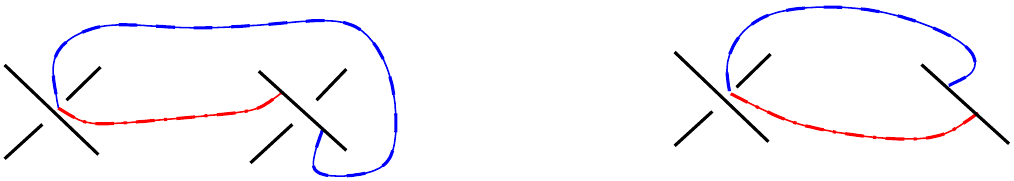}
  \caption{A triangle with one edge on red, one on blue, and one on the knot}
  \label{fig:redblue-corner}
\end{figure}

Consider the left of Figure~\ref{fig:redblue-corner}.  By connecting the endpoints of arcs on the link, and pushing the triangle into the plane of projection at that point, we obtain a red--blue bigon.  Arguing as in the proof of Lemma~\ref{lemma:red-blue-bigons}, we conclude there is a contradiction to primality.  Thus if there is a triangle, it must lie in the diagram as on the right of Figure~\ref{fig:redblue-corner}.

As for the figure on the right, we may connect the endpoints of the arcs to form a closed curve $\gamma$ meeting the diagram twice.  We will use Lemma~\ref{lemma:K2-0-prime} to show there are no crossings in the interior of $\gamma$.  If there are such crossings, then the red arc of $\gamma$ must run through a crossing circle, and since the triangular region bounds a disk disjoint from the crossing circle, the blue arc of $\gamma$ must also run through the crossing disk.  Then that crossing disk splits $\gamma$ into two closed curves, which can be pushed to meet the diagram twice, meeting one fewer crossing circle.  By induction on the number of crossing circles met by such a curve, we conclude there are no crossings of $K_{B,i}$ in the interior of $\gamma$.

Thus the red and blue arcs in that figure are in fact parallel to a single strand of the link.  Hence they are both homotopic to a portion of the arc running from the top of the crossing shown in that figure to its base.  Use this homotopy to slide the image of $D$ to this crossing arc, removing the intersection of the red and blue arcs and removing the triangle.
The result has one fewer vertex of $\Gamma_{BR}$, and does not affect the number of vertices of $\Gamma_B$, contradicting our assumption that complexity is as small as possible.
\end{proof}

Since red arcs cannot intersect red arcs, Lemma~\ref{lemma:htpc-opparcs} implies that either all red arcs run from north to south, or all red arcs run from east to west.  

Suppose all blue arcs and all red arcs run from north to south.  Because there are no red--blue bigons by Lemma~\ref{lemma:red-blue-bigons}, all such edges are disjoint.  Then either all edges are blue, and there are no red edges of intersection at all, or there is a rectangle with one blue side and one red side, with north and south edges on $K$, and with interior disjoint from the red and blue surfaces.  The next lemma deals with the latter case.

\begin{lemma}\label{lemma:htpc-noredbluerect}
If the graph $\Gamma_{BR}$ cuts $D$ into a subrectangle with two opposite sides mapped to $\partial N(K)$, one side on blue, one side on red, and interior disjoint from blue and red, then the blue and red sides of that rectangle are homotopic to the same crossing arc, and the homotopies can be taken to lie entirely in the blue and red surfaces, $S_{B,i}$ and $R_i$, respectively.
\end{lemma}

Recall that a \emph{crossing arc} is an arc in the link complement that runs straight from the top of a crossing to the bottom in the diagram.

\begin{proof}
Consider such a rectangle.  Abuse notation slightly and give the blue arc on the west side the label $e_1$, and the red arc on the east the label $e_2$, and call the rectangle $D$.  Because the interior of $D$ does not meet blue or red surfaces, it can be mapped into the complement of $N(K_{B,i})$, missing checkerboard surfaces, hence it must be mapped completely above or below the plane of projection of $K_{B,i}$. Without loss of generality, assume it is mapped above.  Then the arcs of $\bdy D$ on $\bdy N(K_{B,i})$ each lie on a single strand of the diagram, i.e.\ on a strand running between two undercrossings. Endpoints of $e_1$ and $e_2$ either straddle an overcrossing or lie on either side of a strand.

There are three cases to consider: first, endpoints of $e_1$ and $e_2$ straddle overcrossings at both ends; second, one set of endpoints straddles an overcrossing and one set lies on opposite sides of a strand; and third, both sets of endpoints lie on opposite sides of strands.  The cases are shown in Figure~\ref{fig:redblue-rect2}.

\begin{figure}
  \includegraphics{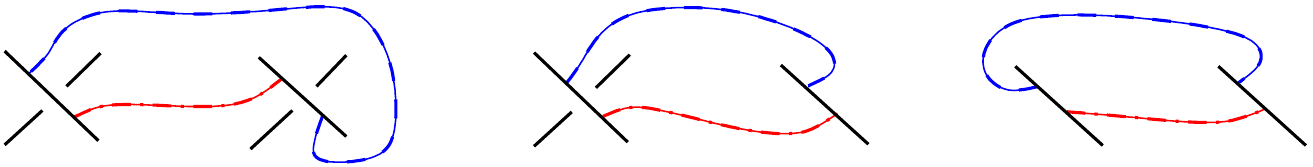}
  \caption{Possible images of red--blue rectangles}
  \label{fig:redblue-rect2}
\end{figure}

In the first case, we may connect red and blue arcs to form a closed curve meeting the diagram twice at crossings.  As in the proof of Lemma~\ref{lemma:red-blue-bigons}, this gives a contradiction.

In the second and third case, we connect red and blue arcs to form a closed curve $\gamma$ meeting the diagram twice.  As in the proof of the previous lemma, we may use Lemma~\ref{lemma:K2-0-prime} to argue that there are no crossings of $K_{B,i}$ on one side of $\gamma$.  In the third case, blue and red edges are parallel and not essential, and we can modify $D$ by homotoping off this region of the diagram,
removing these intersections with red and blue surfaces, and reducing complexity, contradicting our minimality assumption.  In the second case, $e_1$ and $e_2$ are homotopic to the same crossing arc, as desired.  Notice that the homotopies taking these arcs to the crossing arc lie entirely in the blue or red surface, respectively.
\end{proof}

\begin{lemma}\label{lemma:htpc-northsouth-red}
If the graph $\Gamma_{BR}$ consists of disjoint red and blue arcs on $D$, all running north to south, then $e_1$ and $e_2$ are each homotopic in the blue surface to arcs in the same subsurface associated with a twist region of $K_i$.
\end{lemma}

\begin{proof}
We will show that in this case, each arc of $\Gamma_{BR}$ is homotopic, in the surface $S_{B,i}$ or $R_i$ containing it, into a subsurface corresponding to a twist region of $K_i$. We will also show that successive arcs lie in subsurfaces corresponding to the same twist region. Now, two arcs, both lying in $S_{B,i}$ or both lying in $R_i$, and lying in subsurfaces corresponding to distinct twist regions, cannot be homotopic in those surfaces. Hence, we will deduce that $e_1$ and $e_2$ are homotopic in $S_{B,i}$ to the subsurface corresponding to the same twist region.

In the case where successive arcs are red and blue, Lemma~\ref{lemma:htpc-noredbluerect}, applied to the part of $D$ lying between these arcs, gives this claim.
Thus we need to show that when there are two adjacent blue arcs of $\Gamma_{BR}$, or two adjacent red arcs, the result still holds.  

In the blue case, we have a rectangle with two sides on blue, two sides on $\bdy N(K_{B,i})$, and interior disjoint from red and blue.  It must be mapped by $\phi$ entirely to one side of the projection plane.  Its edges on $N(K_{B,i})$ must run over crossings, else we could homotope $D$ to remove both intersections, contradicting our requirement that $\Gamma_{BR}$ be minimal. Hence it defines a simple closed curve $\gamma$ in the diagram meeting the knot in exactly two crossings.  Because $K_{B,i}$ is blue twist reduced by Lemma~\ref{lemma:blue-twist-reduced}, the curve $\gamma$ encircles a collection of red bigons.  Hence we may isotope the blue arcs on $S_{B,i}$, relative to their endpoints on $\bdy N(K_{B,i})$,
to lie in a neighborhood of the twist region of $K_{B,i}$ containing those red bigons.
Note this is a twist region of $K_i$ as well, since the blue surfaces lie on the outside of the twist region, so the result follows in this case.

In case that there are two adjacent red arcs in $D$, the argument is similar.  Again the rectangle must be mapped entirely to one side of the projection plane of $K_{B,i}$, and it defines a simple closed curve $\gamma$ in the diagram meeting the knot in exactly two crossings, with arcs $\gamma_1$ and $\gamma_2$ in the red surface running between these crossings.  The diagram of $K_{B,i}$ may not be red twist reduced, which means $\gamma$ is not required to bound blue bigons on one side, a priori.  However, by induction we may assume that one of the red arcs, say $\gamma_1$, is homotopic to a crossing arc in a twist region of $K$.

The arc $\gamma_1$ has endpoints on overstrands of distinct crossings. Since $\gamma_1$ is homotopic in $R_i$ to a crossing arc of a single crossing, one of the overstrands must run directly to the understrand of the other crossing, and the arc must be homotopic to that strand of the knot.  Then we may slide $\gamma$ in the diagram of $K_{B,i}$ to the opposite side of this strand, forming a closed curve $\alpha$ consisting of $\gamma_2$ and an arc parallel to the knot strand parallel to $\gamma_1$, and meeting the diagram twice.
This closed curve $\alpha$ must be disjoint from the crossing circles of $L_{B,i}$, because any such intersection would imply that $\gamma$ linked a crossing circle, and hence that $D$ contained a blue vertex corresponding to this crossing circle, which is a contradiction.

Moreover, $\alpha$ must be disjoint from all the crossing disks, for if it were to intersect a crossing disk, it would do so twice, once in the blue surface and once in the red. Then $\gamma$ would also intersect this crossing disk twice, once in $\gamma_1$, and once in $\gamma_2$. We may then form a simple closed curve in the diagram for $L_{B,i}$ that starts at a crossing where $\gamma_1$ and $\gamma_2$ meet, runs along $\gamma_1$ as far as the crossing disk, runs along the crossing disk to $\gamma_2$ and then back along $\gamma_2$ to the original crossing. By choosing the crossing disk appropriately, we may ensure that this curve intersects no other crossing disks. After a small isotopy, it can then be made disjoint from the crossing disks, and so it then specifies a simple closed curve in the diagram for $K$ which is disjoint from $K$ except at a single crossing. This implies that $K$ was not prime, which is a contradiction.

Hence, $\alpha$ corresponds to a simple closed curve in the diagram of $K$ that hits $K$ twice. By the primality of $K$, it bounds a region of the diagram with no crossings. Therefore, $\gamma_2$ is homotopic in $R_i$ to a crossing arc, as required.
\end{proof}

It follows from Lemma~\ref{lemma:htpc-northsouth-red} that if $D$ meets the red surface in vertical arcs, then Theorem~\ref{thm:htpcarcs} holds.  So we assume that $D$ does not meet the red surface in vertical arcs.
Then the blue surface cuts $D$ into rectangles with north and south sides on $K$ and east and west sides on $S_{B,i}$, and interiors disjoint from blue.  Because each rectangle is disjoint from all vertices (crossing circles), it can be embedded in $S^3\setminus K_{B,i}$.  The embedding will put east and west sides of the boundary of the rectangle on the blue checkerboard surface of $K_{B,i}$, north and south sides on the link $\partial N(K_{B,i})$, and will map the interior into the complement of the blue checkerboard surface in $S^3\setminus {\rm int}(N(K_{B,i}))$.  A rectangle embedded in $S^3\setminus {\rm int}(N(K_{B,i}))$ in this way is a well--known object: it is an \emph{essential product disk}.

\begin{define}
An \emph{essential product disk} for the blue checkerboard surface $B$ of a knot $K$ is an essential disk properly embedded in $S^3\setminus {\rm int}(N(B))$, whose boundary is a rectangle with two opposite sides on $N(B)$ and two opposite sides on $\bdy N(K)$.
\end{define}

Essential product disks have been studied in many other contexts (for example, to identify the guts of a manifold \cite{gabai-kazez, lackenby:volume-alt, fkp:gutsjp}).

To conclude the proof of Theorem~\ref{thm:htpcarcs}, we will consider essential product disks for $K_{B,i}$, and show the boundary of such a disk gives two arcs in the neighborhood of a twist region in the diagram of $K_i$.

\begin{lemma}\label{lemma:htpc-epd}
Let $e_1$ and $e_2$ denote the boundary arcs on the blue surface in an essential product disk for the blue checkerboard surface of $K_{B,i}$.  Then there is a subsurface associated with a twist region of the diagram of $K_i$, and arcs $a_1$ and $a_2$ in that subsurface, such that $e_1$ is homotopic to in $S_{B,i}$ to $a_1$, and $e_2$ is homotopic in $S_{B,i}$ to $a_2$.
\end{lemma}

\begin{proof}
Let $E$ denote the essential product disk.  If $E$ is disjoint from the red checkerboard surface for $K_{B,i}$, then an argument as in the proof of Lemma~\ref{lemma:htpc-northsouth-red}, using the fact that the diagram of $K_{B,i}$ is blue twist reduced, implies the blue edges of $E$ are homotopic to arcs in a subsurface associated to a single twist region of $K_i$.

So suppose the essential product disk $E$ meets the red surface. By Lemma~\ref{lemma:htpc-opparcs}, we may assume intersections with the red surface do not run from a blue edge to $N(K_{B,i})$.  By Lemma~\ref{lemma:htpc-noredbluerect}, if intersections with red have both endpoints on $N(K_{B,i})$, then the blue edges of $E$ are both homotopic to the same crossing arc, hence can be homotoped to lie in the same twist region of the diagram of $K_i$.  Hence we assume the red surface meets $E$ in a sequence of horizontal arcs, cutting it into rectangles.

It is well known that the checkerboard surfaces of a connected alternating link diagram cut the complement into two identical 4--valent ideal polyhedra.  These two polyhedra have edges corresponding to edges of the diagram of the knot, and ideal vertices corresponding to vertices of the knot.  The knot complement is obtained by gluing the same red faces of the two polyhedra by a single clockwise rotation.  Blue faces are glued by a counter clockwise rotation.  See, for example, \cite{menasco:polyhedra} or \cite{lackenby:volume-alt}.  

Our sequence of rectangles making up $E$ has boundary components which lie in the checkerboard surfaces of $K_{B,i}$.  The sequence alternates lying in one polyhedron and then the other, but their boundaries can be sketched into the diagram graph of the knot $K_{B,i}$, using the identification of edges and vertices of the polyhedron with edges and vertices of the diagram.  

Consider the rectangle at the north of $E$, denote it by $E_1$, and the rectangle glued just under it, $E_2$.  The rectangle $E_1$ at the north has one side running along $N(K_{B,i})$, which means it has a side cutting through an ideal vertex of a polyhedron.  Push this off the ideal vertex slightly, to cut off a single vertex of a red face.  Now consider its side of $E_1$ in the other red face.  This is glued by a clockwise turn to a side of $E_2$.  Impose the image of $E_2$ under this clockwise turn on the polyhedron containing $E_1$, and
denote it by $\bar{E}_2$.  By \cite[Lemma~7]{lackenby:volume-alt} (see also \cite[Lemma~4.9]{fkp:gutsjp}), if these sides of $E_1$ and $\bar{E}_2$ intersect in this red face, then they must intersect in two red faces.  But the side of $E_1$ cuts off a single vertex in its other red face, so it can be homotoped such that it does not intersect another side of $\bar{E}_2$.  Hence $E_1$ and $\bar{E}_2$, and thus $E_2$, must each cut off a single vertex in the red face in which they are glued.  The same argument shows that $E_2$ and $E_3$ also both cut off a single vertex in the red face in which they are glued.  By
induction, each rectangle making up $E$ has sides in red faces cutting off a single ideal vertex.

These rectangles map to regions of the diagram meeting the diagram exactly four times, adjacent to two crossings.  We can push the sides in the red faces onto these crossings.  Because the diagram is blue twist reduced by Lemma~\ref{lemma:blue-twist-reduced}, each such rectangle bounds a (possibly empty) string of red bigons.  Hence their boundaries all lie in a neighborhood of the same twist region of the diagram of $K_{B,i}$. Note in this case, the twist region must also be a twist region of $K_i$, since the blue surface is outside the twist region. 
\end{proof}

We can put this together to finish proof of Theorem~\ref{thm:htpcarcs}.

\begin{proof}[Proof of Theorem~\ref{thm:htpcarcs}]
Lemma~\ref{lemma:htpc-graph} implies that two edges $e_1$ and $e_2$ that are homotopic in $S^3\setminus K$ give rise to a mapping of a disk $\phi\co D\to S^3\setminus K$ and a graph $\Gamma_B$.  Provided $N_\tw \geq 72$ if $i=0$, and $N_\tw \geq 121$ if $i=2$, the graph $\Gamma_B$ contains no blue vertices, by Lemma~\ref{lemma:htpc-novertices}.  By Lemmas~\ref{lemma:htpc-straightarcs} and~\ref{lemma:htpc-opparcs}, blue edges of $\Gamma_{BR}$ run north to south, and red edges either all run north to south or all run east to west.  If red run north to south, then Lemma~\ref{lemma:htpc-northsouth-red} implies that $e_1$ and $e_2$ are homotopic to arcs in the same subsurface associated with a twist region.  If red run east to west, then Lemma~\ref{lemma:htpc-epd} implies they are homotopic to arcs in the same subsurface associated with a twist region of $K_i$.
\end{proof}

We finish with a result for regular checkerboard surfaces that follows almost immediately from the previous work.

\begin{prop}\label{prop:htpc-checkerboard}
Let $S$ denote the disjoint union of the two checkerboard surfaces of a prime, twist reduced alternating diagram of a hyperbolic knot $K$.  Suppose $a_1$ and $a_2$ are disjoint essential embedded arcs in $S$ that are not homotopic in $S$, but are homotopic in $S^3\setminus K$ after including $S$ into $S^3\setminus K$.  Then either $a_1$ and $a_2$ are isotopic in $S$ to crossing arcs in the same twist region of the diagram, or they both lie on the same checkerboard surface and both are isotopic in that checkerboard surface to arcs in the same subsurface associated with some twist region of $K$.
\end{prop}

\begin{proof}
Given a prime, twist reduced alternating diagram $K$, let $C$ denote the maximal number of crossings in any twist region of $K$.  Let $N_\tw$ be the maximum of $121$ and $C$.
Then the diagram $K_2$ obtained from that of $K$ by removing pairs of crossings from each twist region of $K$ with more than $N_\tw$ crossings is identical to the diagram of $K$.  Moreover, the diagrams of $K_i$ and $K_{B,i}$ are also identical, and the surfaces $R_i$ and $S_{B,i}$ are identical to the checkerboard surfaces of $K$.  We therefore apply the above results to these knots and these surfaces.  In particular, Theorem~\ref{thm:htpcarcs} implies that if $e_1$ and $e_2$ lie on the same surface, either red or blue, then the two arcs are homotopic in that surface into the same subsurface associated with a twist region.  Thus we only need to finish the case that $e_1$ and $e_2$ lie on different surfaces, say $e_1$ on blue and $e_2$ on red.

First, modify Lemma~\ref{lemma:htpc-graph} in a straightforward way to allow these two arcs to lie on distinct surfaces.  We obtain a graph $\Gamma_{BR}$, which we assume, as usual, has minimal complexity.  There will be no vertices at all in the graph $\Gamma_{BR}$, since there are no crossing circles added to adjust the diagram of $K$, so Lemma~\ref{lemma:htpc-novertices} trivially holds.  We may argue just as in Lemma~\ref{lemma:htpc-straightarcs} that no red or blue arc has both endpoints on the same side of $D$ (north, south, east, or west).  By considering the same triangles of Lemma~\ref{lemma:htpc-opparcs}, or those with red and blue switched, we may argue that red and blue arcs either run north to south, or east to west, just as in the conclusion of that lemma.  However, because we have two distinct colors on east and west, and because the red and blue surfaces are embedded, no arc may run east to west.  Thus all arcs run north to south, and they cut $D$ into rectangles with two sides on $\bdy N(K)$, and two sides on red or blue, and interior disjoint from these surfaces.

If there are only the two original edges, and $\Gamma_{BR}$ is disjoint from the red and blue surfaces otherwise, then we are done by Lemma~\ref{lemma:htpc-noredbluerect}.  Similarly, if the edges of $\Gamma_{BR}$ alternate red and blue, then each subrectangle gives a homotopy to the same crossing arc.  Since $e_1$ and $e_2$ are on different surfaces, there must be at least one subrectangle with sides on different surfaces, so Lemma~\ref{lemma:htpc-noredbluerect} implies that the two arcs of this subrectangle are homotopic to a crossing arc, with the homotopy taken within the respective surface.  

Now suppose there is a subrectangle with both sides on the same (red or blue) surface.  By induction, we may assume that one of the arcs is homotopic to a crossing arc in its surface.  Arguing as in Lemma~\ref{lemma:htpc-northsouth-red}, the two arcs on either side of the rectangle will be homotopic to arcs which together encircle a twist region.  Since one of those arcs is homotopic to a crossing arc, that arc runs from one side of a single crossing of that twist region to the other.  Hence both arcs encircle a trivial twist region, consisting of one crossing, and the other arc is also homotopic to the same crossing arc.  Putting this all together implies the result.  
\end{proof}

\bibliographystyle{amsplain}
\bibliography{references}

\end{document}